\documentclass[a4paper,11pt]{amsart}
\usepackage{amssymb}

\usepackage[utf8]{inputenc}
\usepackage{amssymb}
\usepackage{amsfonts}
\usepackage{amsmath}
\usepackage{amsthm}

\usepackage[abbrev,nobysame]{amsrefs}

\usepackage[colorlinks=true]{hyperref}  

\newtheorem{thm}{Theorem}
\newtheorem{thm*}{Theorem}[subsection]
\newtheorem{lem}{Lemma}[section]
\newtheorem{prop}{Proposition}[section]
\newtheorem{cor}{Corollary}[section]

\newtheorem{conj}{Conjecture}

\theoremstyle{definition}
\newtheorem{defn}{Definition}[section]

\theoremstyle{remark}
\newtheorem{rem}{Remark}[section]

\makeatletter \@addtoreset{equation}{section}

\newcommand{\thmref}[1]{Theorem~\ref{#1}}

\newcommand{\lemref}[1]{Lemma~\ref{#1}}

\def\a{\alpha}
\def\b{\beta}

\def\p{\partial}
\def\vphi{\varphi}

\def\L{\Lambda}

\def\cL{{\cal L}}

\newcommand{\tr}{\mathrm{tr}}

\def\cL{{\mathcal L}}

\def\and{\quad{\rm and}\quad}

\def\eps{\epsilon}

\let\lra=\longrightarrow

\def\mapright{\xrightarrow}
\def\mapright\#1{\,\smash{\mathop{\lra}\limits^{\#1}}\,}

\def\om{\omega}
\def\Om{\Omega}

\def\tri{\triangle}

\newcommand{\iddbar}{\sqrt{-1}\partial\bar{\partial}}  
\numberwithin{equation}{section}

\makeatletter

\newcommand{\Rmnum}[1]{\expandafter\@slowromancap\romannumeral #1@}

\newcommand{\Tr}{\mathrm{tr}}
\newcommand{\Vol}{\mathrm{Vol}}
\newcommand{\rk}{\mathrm{rk}}

\allowdisplaybreaks[1]     
\makeatletter
\newcommand{\pushright}[1]{\ifmeasuring@#1\else\omit\hfill$\displaystyle#1$\fi\ignorespaces}  
\newcommand{\pushleft}[1]{\ifmeasuring@#1\else\omit$\displaystyle#1$\hfill\fi\ignorespaces}
\makeatother

\title[{\tiny Construction of cscK cone metrics}]{Construction of constant scalar curvature K\"ahler cone metrics}
 \author{Julien Keller} \address{Julien Keller \\ Aix Marseille Universit\'e, CNRS, Centrale Marseille, Institut de Math\'ematiques
de Marseille, UMR 7373, 13453 Marseille, France}
 \email{julien.keller@univ-amu.fr}

\author{Kai Zheng}
  \address{Kai Zheng \\ Mathematics Institute, University of Warwick, Coventry, CV4 7AL, UK}
  \email{K.Zheng@warwick.ac.uk}

\date{\today}
\keywords{constant scalar curvature K\"ahler metrics, cone singularities, Hermitian-Einstein, projective bundles, log K-stability}
\begin{document}

\begin{abstract}
Over a compact K\"ahler manifold, we provide a Fredholm alternative result for the Lichnerowicz operator associated to a K\"ahler metric with conic singularities along a divisor. We deduce several existence results of constant scalar curvature K\"ahler metrics with conic singularities: existence result under small deformations of K\"ahler classes, existence result over a Fano manifold, existence result over certain ruled manifolds. In this last case, we consider the projectivisation of a parabolic stable holomorphic bundle. This leads us to prove that the existing Hermitian-Einstein metric on this bundle enjoys a regularity property along the divisor on the base. 
\end{abstract}
\maketitle

\section{Introduction}

In this paper we investigate the construction of constant scalar curvature K\"ahler metrics (cscK in short) with conical singularities over a smooth compact K\"ahler manifold and provide several existence results. 
Starting with a model metric $\omega_D$ with conical singularity along a smooth divisor $D$ of the compact K\"ahler manifold $X$ 
(see definition in Section \ref{Model cone metric}), we are interested in cscK cone K\"ahler metrics $\om$, i.e metrics of the form $\om=\om_D+i\p\bar\p\vphi$ such that \begin{itemize}
\item $\om$ is a K\"ahler cone metric,
\item $\om$ has constant scalar curvature over the regular part $M:=X\setminus D$.
\end{itemize}
In our study, we will consider the linearization of the constant scalar curvature equation. This leads to consider the Lichnerowicz operator over functions $u$ defined by
\begin{align*}
{{\mathbb{L}\mathrm{ic}}}_{\om}(u)=\tri_{\om}^2 u + u^{i\bar j}R_{i\bar j}(\om)
\end{align*}
and the associated {\it Lichnerowicz equation}
\begin{align}\label{Liceqn}
{{\mathbb{L}\mathrm{ic}}}_{\om}(u)=f
\end{align}
for $f\in C^{,\a,\b}$ with $\int_M f\om^n=0$, $n$ being the complex dimension of the manifold and $\om$ defined as above. Note that our study will require to work with certain H\"older spaces adapted to the singularities, the spaces $C^{.,\a ,\b}$, that are described in details in Section  \ref{Model cone metric}.  In particular, we say a K\"ahler potential $\vphi$ is $C^{2,\a,\b}$  cscK cone potential (resp. 
$C^{4,\a,\b}$ cscK cone potential) if $\om=\om_D+i\p\bar\p\vphi$ is a cscK cone metric and additionally $\vphi\in C^{2,\a,\b} $ (resp. $C^{4,\a,\b}$). In the sequel when we speak of cscK cone metric, its potential is at least $C^{2,\a,\b}$ as in \cite[Definition 2.9]{LiL}.

We will need a certain restriction on the cone angle $2\pi\b$ and the H\"older exponent $\alpha$, namely that
\begin{align}\label{anglerestriction} \tag{C}
0<\b<\frac{1}{2};\quad \a\b<1-2\b.
\end{align}
This restriction appeared in previous works, e.g. in \cites{MR3144178,LZ2} and is required to have regularity results.
\medskip

Our first theorem is an analytic result Fredholm alternative  type. It provides  a solution to the Lichnerowicz equation \eqref{Liceqn} over the H\"older spaces $C^{4,\a,\b}$ for $C^{,\a,\b}$ data.
\begin{thm}[Linear theory]\label{Fredholm}
Let $X$ be a compact K\"ahler manifold, $D\subset X$ a smooth divisor, $\omega$ a cscK cone metric with $C^{2,\a,\b}$ potential such that the cone angle $2\pi\beta$  and the H\"older exponent $\alpha$ satisfy Condition \eqref{anglerestriction}. Assume that $f\in C^{,\a,\b}$ with normalisation condition $\int_M f\om^n =0$. Then one of the following holds:
\begin{itemize}
\item Either the Lichnerowicz equation ${{\mathbb{L}\mathrm{ic}}}_{\om}(u)=f$ has a unique $C^{4,\a,\b}$ solution.
\item Or the kernel of ${{\mathbb{L}\mathrm{ic}}}_{\om}(u)$ generates a holomorphic vector field tangent to $D$.
\end{itemize}
\end{thm}
Note that the solution furnished by the theorem can be extended continuously to the whole manifold $X$.

As in the smooth situation (see for instance Lebrun-Simanca's results in \cite[Corollary 2]{MR1274118}), this result provides an existence theorem by small deformations. To derive it, we just use the implicit function theorem together with the one-one correspondence between the kernel of the Lichnerowicz operator and the holomorphic vector fields on the manifold tangential to the divisor, see \cite[Section 4]{LZ2}. We introduce some notations. Set $Aut(X)$ the group of holomorphic transformations of $X$ given by diffeomorphisms of $X$ that preserve the complex structure. We consider the subgroup $Aut_D(X,[\omega])\subset Aut(X)$ as the identity component of the automorphisms group that preserve the K\"ahler class $[\omega]$ and fix the divisor $D$. Then, $Lie(Aut_D(X,[\omega]))$ consists in the Lie algebra of holomorphic vector fields tangential to $D$ with holomorphy potential. Recall that a holomorphy potential is a function whose complex gradient, with respect to the metric $\omega$  is a holomorphic vector field. 

\begin{cor}[CscK cone metric by deformation]\label{cor1}
Consider $(X,\omega)$ compact K\"ahler manifold endowed with $\omega$ a cscK cone metric along $D\subset X$,  smooth divisor, with angle $\beta$ satisfying Condition (\ref{anglerestriction}). Assume that the Lie algebra $Lie(Aut_D(X,[\omega_B]))$ is trivial. Then the set of all K\"ahler classes around $[\omega]$ containing a cscK metric with cone singularities is non-empty and open.
\end{cor}

A direct application of this last result  is the existence of cscK cone metrics close to K\"ahler-Einstein cone metrics on Fano manifolds. Before stating the result, we refer to \cites{MR2484031,MR3107540} for a definition the $\alpha$-invariant for general polarization and its relation with log-canonical thresholds. The next corollary is obtained from the results of Berman \cite{MR3107540} and Li-Sun \cite[Corollaries 2.19 and 2.21]{MR3248054} on existence of a K\"ahler-Einstein cone metric over a Fano manifold and the non existence of holomorphic vector field tangent to $D$ (when the parameter $\lambda$ below is greater or equal to $1$). The regularity of the K\"ahler-Einstein cone metric is also sufficient to apply Theorem \ref{Fredholm} (the regularity issue is discussed  in \cite{LZ}, see also references therein). 
\begin{cor}[CscK cone metrics for Fano manifolds]
Assume that $\Om_0=c_1(-K_X)$ and $D$ is a smooth divisor which is $\mathbb{Q}$-linearly equivalent to $-\lambda K_X$, where $\lambda\in \mathbb{Q}_+^*$. Denote $L_D$ the line bundle associated to $D$. 
\begin{enumerate}
 \item[(i)] If $\lambda\geq 1$, then there is a constant $\delta$ such that in the nearby class $\Om$ satisfying $|\Om-\Om_0|<\delta$ there exists a constant scalar curvature K\"ahler cone metric in $\Om$ with cone angle $2\pi\beta$ satisfying \begin{align} \label{cond} 0<\b<\min \left(\frac{1}{2}, \left(1-\frac{1}{\lambda}\right)+\frac{n+1}{n}\min\left(\frac{1}{\lambda}\alpha(-K_X),\alpha({L_D}_{\vert D})\vert_D\right)\right).\end{align}
\item[(ii)] If $\frac{2n}{2n+1}<\lambda <1$, $Lie(Aut_D(X,\Omega_0))$ is trivial, 
then the same conclusion as in (i) holds for angle $2\pi\beta$ satisfying \eqref{cond} and the extra condition $\beta>n\left(\frac{1}{\lambda}-1\right).$
\end{enumerate}
\end{cor}
Note that the upped bound in \eqref{cond} may not be optimal but has the advantage of being effective and calculable, we also refer to \cite{MR3470713} on that point.

Our next main result is a construction theorem of  cscK metrics with cone singularities in K\"ahler classes (that may not be integral) over projective bundles, which generalizes the main result of \cite{Hong2}. It is also an application of Theorem \ref{Fredholm} but requires much more work. The notion of parabolic stability is explained in Sections \ref{parabolicstab} and \ref{coneHE}. 
\begin{thm}[CscK cone metric for projective bundles]\label{thm1}
Let $B$ be a base compact K\"ahler manifold endowed with a cscK metric $\omega_B$ with cone singularities  along $D\subset B$, smooth divisor with trivial Lie algebra $Lie(Aut_D(B,[\omega_B]))$. Assume the H\"older exponent $\a$ and the angle $2\pi\b$ of $\omega_B$ satisfy Condition \eqref{anglerestriction}. Let $E$ be a parabolic stable vector bundle over $B$ with respect to $\omega_B$. \\
Then, for $k\in\mathbb{N}^*$ large enough, there exists  a cscK metric with cone singularities on $X:=\mathbb{P}E^*$ in the class $$\Om=[k\pi^*\omega_B+\hat{\omega}_E]$$
where $\pi:X\rightarrow B$ and $\hat{\omega}_E$ represents the first Chern class of $\mathcal{O}_{\mathbb{P}E^*}(1)$. This cscK metric has its cone singularities along $\mathcal{D}:=\pi^{-1}(D)$ with $C^{4,\a,\b}$ potential.
\end{thm}
\begin{rem}
 In Theorem \ref{thm1}  the assumption on $E$ could be replaced by saying that it is an indecomposable holomorphic vector bundle equipped with a parabolic structure and a Hermitian-Einstein cone metric compatible with this structure, providing a purely differential geometric statement.
\end{rem}

Let's do now some brief comments. CscK cone metrics constitute a natural generalization of K\"ahler-Einstein metrics with conical singularities (see  Section \ref{recentprogress}) 
. The importance of the notion of K\"ahler-Einstein cone metric is now well established from Chen-Donaldson-Sun's breakthrough for the celebrated Yau-Tian-Donaldson conjecture, when one restricts attention to Fano manifolds and the anticanonical class, see for instance the pioneering paper \cite{MR2975584} or the survey \cite{MR3522175} and references therein. One may expect that the study of cscK cone metrics may lead to new progress on Yau-Tian-Donaldson conjecture for general polarizations or may have applications for construction of moduli spaces or Chern number inequalities. 
Furthermore, a logarithmic version of Yau-Tian-Donaldson conjecture is expected to be also true in the context of cscK cone metrics. Nevertheless, as far as we know, only very few examples of cscK cone metrics that are not K\"ahler-Einstein appeared in the literature. CscK cone metrics are far from being well understood and for instance uniqueness results have only appeared very recently, cf. \cites{LZ2,LZ3}. \\
From the point of view of existence, the case of curves has been studied by R.C. McOwen, M. Troyanov and F. Luo--G.Tian in the late eighties. In higher dimension, Y. Hashimoto \cite{Hasfutaki} has recently obtained momentum-constructed cscK cone metrics on the projective completion of a pluricanonical line bundle over a product of K\"ahler-Einstein Fano manifolds. This enabled him to give first evidence of the log Yau-Tian-Donaldson conjecture. Note that his definition of K\"ahler cone metrics is more restricted than the general usual definition that we consider here. A more general setup has been studied in \cite{Ke1} where it is shown morally that the notion of cscK cone metric is the most natural notion of K\"ahler metrics with special curvature properties for projective bundles over a curve, when the holomorphic bundle is irreducible and not Mumford stable (otherwise, the ``right'' notion would be the classical notions of smooth extremal/cscK metric).  A related work for extremal K\"ahler cone metric, still on the projective completion of a line bundle over admissible manifolds, can also be found in \cite{LiH}.

Our results provide an effective method to construct plenty of cscK cone metrics on various manifolds and partially generalize previous results op. cit. We also expect that Theorem \ref{Fredholm} will have many applications in a long range, including for studying log-K-stability (see Section \ref{remarksection}). 

We shall now explain the structure of the paper. In Section \ref{firstSection}, we introduce the notion of metrics with singularities, together with the adapted H\"older spaces and recall some results about regularity of cscK cone metrics. Among other things we prove the vanishing of the log-Futaki invariant of a K\"ahler class endowed with a cscK cone metric. In Section \ref{Section2}, we introduce weighted Sobolev spaces and obtain Schauder type estimates for Laplacian equation associated to a K\"ahler cone metric under half angle condition (Proposition \ref{Schauder}) or without half angle condition but with weaker regularity  (Proposition \ref{Schauderweakomega}). This allows us to see that a weak solution $u$ to the bi-Laplacian equation $\Delta^2u - K\Delta u =f$ is actually $C^{4,\a,\b}$. Using this result and a continuity method, we are able to prove Theorem \ref{Fredholm} by proving the key estimate (Theorem \ref{closedness}) showing closeness.
In Section \ref{Section 3}, we introduce the notion of {\it Hermitian-Einstein cone metrics}, that are hermitian metrics over a parabolic vector bundle that satisfy the Einstein equation (with respect to a K\"ahler cone metric) together with a certain regularity property. Theorem \ref{HEsol} shows that a parabolic stable vector bundle can be equipped with a Hermitian-Einstein cone metric, refining results of Simpson \cites{Sim,MR1040197} and Li \cite{LiJ}. Using this result, we adapt the work of Hong \cites{Hong2,Hong3} for smooth cscK metrics to the conical setting and construct inductively almost cscK cone metrics (Proposition \ref{prop1}). Using now Theorem \ref{Fredholm} and taking the adiabatic limit, we can deduce Theorem \ref{thm1} in Section \ref{proofThm1}.  In Section \ref{KET}, we explain that the existence of a K\"ahler-Einstein cone metric on a manifold provides a Hermitian-Einstein cone metric on its tangent bundle, generalizing a well-known result in the smooth case. This could be used to provide extra concrete examples of applications of our Theorem \ref{thm1}. Eventually, in Section \ref{remarksection}, we discuss natural generalizations of our work and some possible applications to other geometric questions. In the particular case of the projectivisation of a parabolic vector bundle over a curve, we formulate a conjecture between existence of cscK cone metric, log K-stability and parabolic stability.

\medskip

\tableofcontents
\section{CscK metrics with cone singularities}\label{firstSection}
Let $(X,\om_0)$ be a K\"ahler manifold. We denote $[\om_0]$ the K\"ahler class containing the smooth K\"ahler metric $\om_0$. 
We let $D$ be a smooth divisor in $X$ with $0 <\b <\frac{1}{2}$.

Given a point $p$ in $D$, let $\{z^1,\dots z^k\}$
be the local defining functions of the hypersurfaces where $p$ locates label. The local chart $(U_p, z^i)$ centered at $p$
is called \emph{cone chart} at $p$.

\begin{defn}\label{defn: cone metrics inside introduction}
	
	A \emph{K\"ahler cone metric} $\om$ of cone angle $2\pi\b$ along $D$,
	is a closed positive $(1,1)$ current, which is also a smooth K\"ahler metric on the regular part $$M:=X\setminus D.$$
	{In a local cone chart $U_p$, the} K\"ahler form is quasi-isometric to the standard cone flat metric, which is

	\begin{align}\label{flat cone}
	\om_{cone}
	&:=\frac{\sqrt{-1}}{2}\left(\b^2|z^1|^{2(\b-1)}dz^1\wedge
	dz^{\bar 1}
	+\sum_{2\leq j\leq n}dz^j\wedge dz^{\bar j}\right)\, .
	\end{align}
\end{defn}

The standard cone metric has nice properties.
The Christoffel symbols of the connection of $\om_{cone}$ under the holomorphic coordinate $\{z^1,\dots z^n\}$ are for all $2\leq i,j,k\leq n$,
$$\Gamma^1_{1k}(\om_{cone})=\Gamma^i_{11}(\om_{cone})=\Gamma^1_{jk}(\om_{cone})=\Gamma^i_{1k}(\om_{cone})=\Gamma^i_{jk}(\om_{cone})=0,$$
except $\Gamma^1_{11}(\om_{cone})=-\frac{1-\b}{z^1}.$ Also, the Riemannian curvature of $\om_{cone}$ is identical to zero.

\subsection{H\"older spaces in cone charts}\label{holderspace}

In this section, we start by recalling the definition of Donaldson's H\"older spaces \cite{MR2975584}, see also \cite{MR3144178}.\\
A quasi-isometric mapping $W$ is well defined in the cone chart $U_p$ as follows,
\begin{equation}\label{wcoordiante}
W (z^1, \cdots, \, z^{n})  :=
(w^1=|z^1|^{\b-1}z^1,z^2,\, \cdots, \, \, z^{n})\; .
\end{equation}
We let $v(w^1,\cdots,z^n)=u(z^1,\cdots,z^n)$.
A function $u(z): U_p\rightarrow \mathbb R$ 
is said to be $C^{,\a,\b}$, if $v(w^1,\cdots,z^n)$ is a $C^\a$ H\"older function in the classical sense. The space $C^{,\a,\b}_{\{0\}}$ contains all functions $f\in C^{,\a,\b}$ such that $$f(0, z^2,\cdots, z^n) =0.$$ In the cone charts $U_p$, the H\"older semi-norm $[u]_{C^{,\a,\b}(U_p)}$ is defined to be $[v]_{C^{\a}(W(U_p))}$ and then the H\"older norm $|u|_{C^{,\a,\b}(U_p)}$ is $\sup_{U_p} |u|+[u]_{C^{,\a,\b}(U_p)}$ in the usual sense. The global semi-norm or norm on the whole manifold $X$ is defined by using a partition of unity of $X$, since in the charts away from $D$, everything is defined in the classical sense. Together with the H\"older norm, $C^{,\a,\b}$ becomes a Banach space. A $(1,1)$-form $\sigma$ is said to be $C^{,\a,\b}$, if for any $2\leq i,j\leq n$,
\begin{equation}\label{2ndderivativescone}
\left\{
\begin{aligned}
&\sigma(\frac{\p}{\p z^i},\frac{\p}{\p z^{\bar j}})\in C^{,\a,\b}, &|z^1|^{2-2\b}\sigma(\frac{\p}{\p z^1},\frac{\p}{\p z^{\bar 1}})\in C^{,\a,\b},\\
&|z^1|^{1-\b}\sigma(\frac{\p}{\p z^1},\frac{\p}{\p z^{\bar j}})\in C^{,\a,\b}, &|z^1|^{1-\b}\sigma(\frac{\p}{\p z^i},\frac{\p}{\p z^{\bar 1}}) \in C^{,\a,\b}.
\end{aligned}
\right.
\end{equation} 
Similarly, we could define $C^{,\a,\b}$ of higher order tensors.
The H\"older space $C^{2,\a,\b}$ is defined by
\begin{align*}
C^{2,\a,\b}
= \{u\; \vert \;  u, \p u, \sqrt{-1}\p\bar\p u\in C^{,\a,\b}\}\; .
\end{align*}

Note that the spaces $C^{,\a,\b}$ and $C^{2,\a,\b}$ are independent of the choice of the background K\"ahler cone metrics which are equivalent. But we can see that the higher order spaces are more complicated, since the geometry of the background metric is involved.
The H\"older space $C^{3,\a,\b}$ and $C^{4,\a,\b}$ are introduced in \cites{MR3405866} and further detailed computations can be found in \cite{LZ2}.
The idea is that we first define the local model H\"older spaces in the cone charts, and then extend it to the whole manifold by using a global background K\"ahler cone metric.

We identify $\tilde U$ in the complex Euclidean space as the image of the cone chart $U\subset X$ under the cone chart, i.e a quasi-isometry $\rho: U\rightarrow \tilde U\subset \mathbb C^n$. We call $\tilde U$ an image cone chart. Then we have $C^{,\a,\b}(\tilde U)$ and $C^{2,\a,\b}(\tilde U)$ defined in $\tilde U$ as above with respect to $\om_{cone}$.

\begin{defn}\label{defn: 3dri}
	The H\"older space $C^{3,\a,\b}(\tilde U)$ is defined as the set of function
	$u\in C^{2,\a,\b}(\tilde U)$ such that its 3rd order covariant derivatives with respect to $\nabla^{cone}$ associated to the metric $\om_{cone}$ are $C^{,\a,\b}(\tilde U)$ in an image cone chart $\tilde U$. 
	More precisely, written down with respect to  the standard cone metric $\om_{cone}$, the following covariant derivatives are required to be $C^{,\a,\b}(\tilde U)$,{\small
	\begin{equation}\label{3rdderivativescone}
	\left\{
	\begin{aligned}
	&\nabla^{cone}_{i} \nabla^{cone}_{{\bar l}} \nabla^{cone}_{k}u,
	|z^1|^{3-3\b}\nabla^{cone}_{1} \nabla^{cone}_{{\bar 1}} \nabla^{cone}_{1}u \in C^{,\a,\b}(\tilde U),\\
	&|z^1|^{1-\b}\nabla^{cone}_{i} \nabla^{cone}_{{\bar 1}} \nabla^{cone}_{k}u ,
	|z^1|^{1-\b}\nabla^{cone}_{1} \nabla^{cone}_{{\bar l}} \nabla^{cone}_{k}u \in C^{,\a,\b}(\tilde U),\\
	&|z^1|^{2-2\b}\nabla^{cone}_{i} \nabla^{cone}_{{\bar 1}} \nabla^{cone}_{1}u ,
	|z^1|^{2-2\b}\nabla^{cone}_{1} \nabla^{cone}_{{\bar l}} \nabla^{cone}_{1}u \in C^{,\a,\b}(\tilde U).
	\end{aligned}
	\right.
	\end{equation}} 
\end{defn}
The $C^{4,\a,\b}(\tilde U)$ is defined in the similar way as follows.
\begin{defn}
	We say a $C^{3,\a,\b}(\tilde U)$ function $u$ is $C^{4,\a,\b}(\tilde U)$ if 
	it satisfies for any $2\leq i,j,k,l\leq n$, in the image cone chart $\tilde U$,
	{\small
	\begin{align}\label{4thderivativescone}
	\left\{
	\begin{aligned}
	&|z^1|^{2-2\b}\nabla^{cone}_{\bar l}\nabla^{cone}_k\nabla^{cone}_{\bar j} \nabla^{cone}_i u
	=\frac{\p^4 u}{\p z^{\bar j}\p z^i\p z^{\bar l}\p z^k}\in C^{,\a,\b}(\tilde U),\\
	&|z^1|^{2-2\b}\nabla^{cone}_{\bar l}\nabla^{cone}_1\nabla^{cone}_{\bar j} \nabla^{cone}_i u
	=|z^1|^{1-\b}\frac{\p^4 u}{\p z^{\bar j}\p z^i\p z^{\bar l}\p z^1}\in C^{,\a,\b}(\tilde U),\\
	&|z^1|^{2-2\b}\nabla^{cone}_{\bar 1}\nabla^{cone}_1\nabla^{cone}_{\bar j} \nabla^{cone}_iu
	=|z^1|^{2-2\b}\frac{\p^4 u}{\p z^{\bar j}\p z^i\p z^{\bar 1}\p z^1}\in C^{,\a,\b}(\tilde U),\\
	&|z^1|^{2-2\b}\nabla^{cone}_{\bar l} \nabla^{cone}_1\nabla^{cone}_{\bar j}\nabla^{cone}_1u\\
	&\pushright{=|z^1|^{2-2\b}[\frac{\p^4 u}{\p z^{\bar l}\p z^1\p z^{\bar j}\p z^1}+\frac{1-\b}{z^1}\frac{\p^3 u}{\p z^{\bar l}\p z^1\p z^{\bar j}}]\in C^{,\a,\b}(\tilde U),}\\
	&|z^1|^{3-3\b}\nabla^{cone}_{\bar 1}\nabla^{cone}_1\nabla^{cone}_{\bar 1} \nabla^{cone}_i u\\
	&\pushright{=|z^1|^{3-3\b}[\frac{\p^4 u}{\p z^{\bar 1}\p z^1\p z^{\bar 1}\p z^i}+\frac{1-\b}{z^{\bar 1}}\frac{\p^3 u}{ \p z^1 \p z^{\bar 1}\p z^{ i}}]\in C^{,\a,\b}(\tilde U),}\\
	&|z^1|^{4-4\b}\nabla^{cone}_{\bar 1}\nabla^{cone}_1 \nabla^{cone}_{\bar 1}\nabla^{cone}_1 u\\
	&\pushright{=|z^1|^{4-4\b}[\frac{\p^4 u}{\p z^{\bar 1}\p z^1\p z^{\bar 1}\p z^1}
	+\frac{1-\b}{z^{\bar 1}}\frac{\p^3 u}{\p z^1\p z^{\bar 1} \p z^1}
	+\frac{1-\b}{z^{1}}\frac{\p^3 u}{\p z^1\p z^{\bar 1} \p z^{\bar 1}}}\\
	&\pushright{+\frac{(1-\b)^2}{|z^1|^2}\frac{\p^2 u}{\p z^1\p z^{\bar 1}}]\in C^{,\a,\b}(\tilde U).}
	\end{aligned}
	\right.
	\end{align} }
\end{defn}
We notice from the definitions above that when the same derivative $\p z^1$ or $\p z^{\bar 1}$ appear twice, we need extra lower order derivatives to adjust the normal derivatives.


\subsection{Model cone metric}\label{Model cone metric}
Let $\mathcal H_{\b}(X,D)$ be the space of K\"ahler cone metrics of cone angle $2\pi\b$ along $D$ in the K\"ahler class $[\om_0]$. We denote by $\mathcal H^{2,\alpha,\b}(X,D)$ the space of $\om_0$-plurisubharmonic functions which are $C^{2,\a,\b}(X)$. Let $s$ be a section of associated line bundle of $D$.
It is explained by Donaldson in \cite{MR2975584} that
for sufficiently small $\delta>0$,
\begin{align}\label{model cone}
\om_D=\om_0+\delta \frac{\sqrt{-1}}{2}\p\bar\p|s|^{2\beta}_{h_\L}
\end{align}
is a K\"ahler cone metric and independent of the choices of $\om_0$, $h_\L$, $\delta$ up to quasi-isometry. We call such $\om_D$ the \textbf{model metric} with  conical singularity.
The model metric $\om_D$ has rich geometric information, that we list now. The detailed computation could be found in \cites{MR3144178, MR3405866}.

\begin{lem}\label{geometryg}
	Assume that $0<\b<\frac{1}{2}$.
	The following properties of the model metric $\om_D$ hold in the cone charts.
	\begin{itemize}

		\item For any $2\leq i,j,k\leq n$,
		\begin{equation}\label{bgdmetric}
		\left\{
		\begin{aligned}
		&\frac{\p g_{k\bar l}}{\p z^i}\in C^{,\a,\b},\quad
		|z^1|^{1-\b}\frac{\p g_{1\bar l}}{\p z^i}\in C^{,\a,\b}_{\{0\}},\quad 
		|z^1|^{2-2\b}\frac{\p g_{1\bar 1}}{\p z^i}\in C^{,\a,\b},\\
		&|z^1|^{2-2\b}\nabla_1^{cone}g_{1\bar l}
		=|z^1|^{2-2\b}[\frac{\p g_{1\bar l}}{\p z^1}+\frac{1-\b}{z^1}g_{1\bar l}]\in C^{,\a,\b}_{\{0\}},\\
		&|z^1|^{3-3\b}\nabla_1^{cone}g_{1\bar 1}
		=|z^1|^{3-3\b}[\frac{\p g_{1\bar 1}}{\p z^1}+\frac{1-\b}{z^1}g_{1\bar 1}]\in C^{,\a,\b}_{\{0\}}.
		\end{aligned}
		\right.
		\end{equation} 
		
		\item
		The Christoffel symbols of the connection of $\om_D$ satisfy for any $2\leq i,j,k\leq n$,
		\begin{equation}\label{bgconnection}
		\left\{
		\begin{aligned}
		&\Gamma^{i}_{jk}\in C^{,\a,\b}, \quad 
		|z^1|^{1-\b}\Gamma^{i}_{j1}\in C^{,\a,\b}_{\{0\}}, \quad
		|z^1|^{\b-1}\Gamma^{1}_{jk}\in C^{,\a,\b}_{\{0\}},\\
		&\Gamma^{1}_{j1}\in C^{,\a,\b}, \quad
		|z^1|^{2-2\b}\Gamma^{i}_{11}\in C^{,\a,\b}_{\{0\}},\quad \\
		&|z^1|^{1-\b}\left(\Gamma^{1}_{11}+\frac{1-\b}{z^1}\right)\in C^{,\a,\b}_{\{0\}}.
		\end{aligned}
		\right.
		\end{equation} 
		
		\item For any $2\leq i,j,l\leq n$, the second order covariant derivatives of the model metric $\om_D$ are
		\begin{equation}\label{bgddmetric}
		\left\{
		\begin{aligned}
		&\frac{\p^2 g_{k\bar l}}{\p z^i\p z^{\bar j}}\in C^{,\a,\b},
		|z^1|^{1-\b}\frac{\p^2 g_{1\bar l}}{\p z^i\p z^{\bar j}}\in C^{,\a,\b}_{\{0\}},\\
		&|z^1|^{2-2\b}\frac{\p^2 g_{1\bar 1}}{\p z^i\p z^{\bar j}}\in C^{,\a,\b},\\
		&|z^1|^{2-2\b}\nabla^{cone}_{\bar j}\nabla^{cone}_1 g_{1\bar l}\in C^{,\a,\b}_{\{0\}},\\
		&|z^1|^{3-3\b}\nabla^{cone}_{\bar 1}\nabla^{cone}_i g_{1\bar 1}\in C^{,\a,\b}_{\{0\}},\\
		&|z^1|^{4-4\b}\nabla^{cone}_{\bar 1}\nabla^{cone}_1 g_{1\bar 1}\in C^{,\a,\b}_{\{0\}}.
		\end{aligned}
		\right.
		\end{equation} 
		
	\end{itemize}
\end{lem}

We could define the $C^{3,\a,\b}$ and $C^{4,\a,\b} $ spaces with respect to $\om_D$ on the whole manifold $X$ via replacing the metric $\om_{cone}$ with $\om_D$ in both definitions \eqref{3rdderivativescone} and \eqref{4thderivativescone}. In fact, we have a more general property.
\begin{prop}\label{3rdderivativeD}
	Assume that $\om$ is a K\"ahler cone metric and its connection satisfies \eqref{bgconnection}. 
	Then the local $C^{3,\a,\b}(\tilde U)$ function $u$ could be extended to be global. Precisely, its 3rd order covariant derivatives belong to $C^{,\a,\b}(X)$, i.e. letting $\nabla$ denote the covariant derivative regarding to the K\"ahler cone  metric $\om$, for any $2\leq i,j,k,l\leq n$, 
	\begin{equation}\label{3rdderivativesany}
	\left\{
	\begin{aligned}
	&\nabla_{k} \nabla_{{\bar j}} \nabla_{i}u, &	|z^1|^{3-3\b}\nabla_{1} \nabla_{{\bar 1}} \nabla_{1}u,\\
	&|z^1|^{1-\b}\nabla_{k} \nabla_{{\bar 1}} \nabla_{i}u, &|z^1|^{1-\b}\nabla_{k} \nabla_{{\bar j}} \nabla_{1}u,\\
	&|z^1|^{2-2\b}\nabla_{k} \nabla_{{\bar 1}} \nabla_{1}u, &	|z^1|^{2-2\b}\nabla_{1} \nabla_{{\bar j}} \nabla_{1}u
	\end{aligned}
	\right.
	\end{equation}  belong to  $C^{,\a,\b}(X)$.
\end{prop}
Thus we can define the $C^{3,\a,\b}(X;\om)$ norm of a function $u$ on the whole manifold $X$ as 
\begin{align*}
|u|_{C^{3,\a,\b}(X;\om)}=&|u|_{C^{2,\a,\b}(X)} &\\
&+\sum_{2\leq i,j,l\leq n}[&\pushleft{|\nabla_{k} \nabla_{{\bar j}} \nabla_{i}u|_{C^{,\a,\b}(X)} +|\nabla_{k} \nabla_{{\bar 1}} \nabla_{i}u |_{C^{,\a,\b}(X)}}\vspace{-0.4cm}\\
& &\pushleft{+|\nabla_{k} \nabla_{{\bar j}} \nabla_{1}u |_{C^{,\a,\b}(X)}+|\nabla_{k} \nabla_{{\bar 1}} \nabla_{1}u |_{C^{,\a,\b}(X)}}\\
& &+|\nabla_{1} \nabla_{{\bar j}} \nabla_{1}u |_{C^{,\a,\b}(X)}+|\nabla_{1} \nabla_{{\bar 1}} \nabla_{1}u |_{C^{,\a,\b}(X)}].
\end{align*}

\begin{prop}\label{4thderivativeD}
	Assume that $\om$ is a K\"ahler cone metric and satisfies \eqref{bgconnection}, \eqref{bgdmetric} and \eqref{bgddmetric}. 
	Then the local $C^{4,\a,\b}(\tilde U)$ function $u$ could be extended to be global, i.e. its 4th order covariant derivatives are all $C^{,\a,\b}(X)$, i.e. letting $\nabla$ denote the covariant derivative regarding to the K\"ahler cone  metric $\om$, for any $2\leq i,j,k,l\leq n$,
	\begin{equation*}
	\left\{
	\begin{aligned}
	&\nabla_{\bar l}\nabla_k\nabla_{\bar j} \nabla_i u,
	|z^1|^{1-\b}\nabla_{\bar l}\nabla_1\nabla_{\bar j} \nabla_i u,\\
	&|z^1|^{2-2\b}\nabla_{\bar 1}\nabla_1\nabla_{\bar j} \nabla_i u,
	|z^1|^{2-2\b}\nabla_{\bar l}\nabla_{1} \nabla_{{\bar j}} \nabla_{1}u,\\
	&|z^1|^{3-3\b}\nabla_{\bar 1}\nabla_1\nabla_{\bar 1} \nabla_i  u,
	|z^1|^{4-4\b}\nabla_{\bar 1}\nabla_1 \nabla_{\bar 1}\nabla_1 u
	\end{aligned}
	\right.
	\end{equation*} belong to $C^{,\a,\b}(X)$. 
\end{prop}
The $C^{4,\a,\b}(X;\om)$ norm of a function $u$ is defined in the same way
\begin{align*}
|u|_{C^{4,\a,\b}(X;\om)}=&|u|_{C^{3,\a,\b}(X;\om)} &\\
&\pushleft{+\sum_{2\leq i,j,k,l\leq n}\hspace{-0.15cm}[}& \pushleft{\hspace{-0.5cm}|\nabla_{\bar l}\nabla_k\nabla_{\bar j} \nabla_i u|_{C^{,\a,\b}(X)}+|\nabla_{\bar l}\nabla_1\nabla_{\bar j} \nabla_i u |_{C^{,\a,\b}(X)}}\vspace{-0.4cm}\\
& & \pushleft{ \hspace{-0.5cm}+|\nabla_{\bar 1}\nabla_1\nabla_{\bar j} \nabla_iu |_{C^{,\a,\b}(X)}+|\nabla_{\bar l}\nabla_{1} \nabla_{{\bar j}} \nabla_{1}u|_{C^{,\a,\b}(X)}}\\
& &\pushleft{\hspace{-0.5cm}+|\nabla_{\bar 1}\nabla_1\nabla_{\bar 1} \nabla_i  u|_{C^{,\a,\b}(X)}+|\nabla_{\bar 1}\nabla_1 \nabla_{\bar 1}\nabla_1  |_{C^{,\a,\b}(X)}].}
\end{align*}

The proofs of Propositions \ref{4thderivativeD} and \ref{3rdderivativeD} are carried out in \cite{LZ2}. Moreover both spaces $C^{3,\a,\b}(X;\om)$ and $C^{4,\a,\b}(X;\om)$ are Banach spaces, as proved in \cite[Section 5]{LZ2}.

\begin{rem}
	It is natural to use the model metric $\om_D$ in both  Propositions \ref{3rdderivativeD} and \ref{4thderivativeD}, since it satisfies all required conditions. The same scheme of ideas allows  to define without difficulty higher order function spaces $C^{k,\a,\b}$ for any $k\geq 5$. 
\end{rem}

Next we consider the general K\"ahler cone metric $$\om_\vphi=\om_D+\sqrt{-1}\p\bar\p \vphi.$$
When we have a K\"ahler potential $\vphi\in C^{4,\a,\b}(X;\om_D)$, we get $\log\om^n_\vphi\in C^{2,\a,\b}(X)$, and the Ricci curvature $-\sqrt{-1}\p\bar\p\log\om^n_\vphi\in C^{,\a,\b}(X)$,
according to the definition above and  \lemref{geometryg}, under the restriction \eqref{anglerestriction}. Furthermore, we have the following information for the connection and curvature of $\om_\vphi$.

\begin{cor}\label{cor4ab}
	Assume that the potential function $\vphi$ of a K\"ahler cone metric belongs to $C^{4,\a,\b}(X;\om_D)$. Then the properties \eqref{bgdmetric}, \eqref{bgconnection} and \eqref{bgddmetric} hold for $g_\vphi$. Actually, denoting for simplicity $g=g_\vphi$ as the Riemannian metric associated to $\om_\vphi$, we have that all following elements are in $C^{,\a,\b}$,
	\begin{itemize}
		\item for any $2\leq i,j,k\leq n$,
		\begin{equation*}
		\left\{
		\begin{aligned}
		&\frac{\p g_{k\bar l}}{\p z^i},\quad
		|z^1|^{1-\b}\frac{\p g_{1\bar l}}{\p z^i},\quad
		|z^1|^{2-2\b}\frac{\p g_{1\bar 1}}{\p z^i},\\
		&|z^1|^{2-2\b}\nabla^{cone}_1g_{1\bar l},\quad
		|z^1|^{3-3\b}\nabla^{cone}_1g_{1\bar 1};
		\end{aligned}
		\right.
		\end{equation*} 
		
		\item
		the Christoffel symbols of the connection of $g_\vphi$ for any $2\leq i,j,k\leq n$,
		\begin{equation*}
		\left\{
		\begin{aligned}
		&\Gamma^{i}_{jk}, \quad 
		|z^1|^{1-\b}\Gamma^{i}_{j1},\quad 
		|z^1|^{\b-1}\Gamma^{1}_{jk},\\
		&\Gamma^{1}_{j1}, \quad 
		|z^1|^{2-2\b}\Gamma^{i}_{11}, \quad
		|z^1|^{1-\b}(\Gamma^{1}_{11}+\frac{1-\b}{z^1});
		\end{aligned}
		\right.
		\end{equation*} 
		
		\item for any $2\leq i,j,l\leq n$, the 2nd order covariant derivatives of $g_\vphi$ 
		\begin{equation*}
		\left\{
		\begin{aligned}
		&\frac{\p^2 g_{k\bar l}}{\p z^i\p z^{\bar j}},\quad 
		|z^1|^{1-\b}\frac{\p^2 g_{1\bar l}}{\p z^i\p z^{\bar j}},\quad 
		|z^1|^{2-2\b}\frac{\p^2 g_{1\bar 1}}{\p z^i\p z^{\bar j}},\\
		&|z^1|^{2-2\b}\nabla^{cone}_{\bar j}\nabla^{cone}_1 g_{1\bar l},\quad 
		|z^1|^{3-3\b}\nabla^{cone}_{\bar 1}\nabla^{cone}_i g_{1\bar 1},\\
		&|z^1|^{4-4\b}\nabla^{cone}_{\bar 1}\nabla^{cone}_1 g_{1\bar 1}.
		\end{aligned}
		\right.
		\end{equation*} 
	\end{itemize}
	
\end{cor}

\begin{proof}
	We choose $\om$ to be $\om_{cone}$ in both Proposition \ref{3rdderivativeD} and Proposition \ref{4thderivativeD}. Then the conclusion follows by applying \lemref{geometryg}.
\end{proof}

In addition, if we weaken the condition on the cone potential $\vphi$, we have the following bound.
\begin{cor}\label{cor3ab}
	Assume that the potential function $\vphi$ of a K\"ahler cone metric belongs to $C^{3,\b}(X;\om_D)$. Then the properties \eqref{bgdmetric} and \eqref{bgconnection} hold for $g_\vphi$. Actually, denoting for simplicity $g=g_\vphi$ as the Riemannian metric associated to $\om_\vphi$, we have that all following elements are bounded,
	\begin{itemize}
		\item for any $2\leq i,j,k\leq n$,
		\begin{equation*}
		\left\{
		\begin{aligned}
		&\frac{\p g_{k\bar l}}{\p z^i},\quad
		|z^1|^{1-\b}\frac{\p g_{1\bar l}}{\p z^i},\quad
		|z^1|^{2-2\b}\frac{\p g_{1\bar 1}}{\p z^i},\\
		&|z^1|^{2-2\b}\nabla^{cone}_1g_{1\bar l},\quad
		|z^1|^{3-3\b}\nabla^{cone}_1g_{1\bar 1};
		\end{aligned}
		\right.
		\end{equation*} 
		
		\item
		the Christoffel symbols of the connection of $g_\vphi$ for any $2\leq i,j,k\leq n$,
		\begin{equation*}
		\left\{
		\begin{aligned}
		&\Gamma^{i}_{jk}, \quad 
		|z^1|^{1-\b}\Gamma^{i}_{j1},\quad 
		|z^1|^{\b-1}\Gamma^{1}_{jk},\\
		&\Gamma^{1}_{j1}, \quad 
		|z^1|^{2-2\b}\Gamma^{i}_{11}, \quad
		|z^1|^{1-\b}(\Gamma^{1}_{11}+\frac{1-\b}{z^1});
		\end{aligned}
		\right.
		\end{equation*}

	\end{itemize}
	
\end{cor}

\begin{rem}
In the previous corollary, the same results  hold if we replace $\nabla^{cone}$ by $\nabla^D$.
\end{rem}


\subsection{Second order elliptic equations with conical singularities}\label{Second order elliptic equations with conical singularities}

We first quote a proposition of the general linear elliptic equation which essentially uses Donaldson's estimates \cite{MR2975584} (see also Brendle \cite{MR3144178}, Calamai-Zheng \cite{MR3405866}).
Consider the boundary value problem
\begin{equation}\label{linear equ ge}
\mathbb Lu:=g^{i\bar j}u_{i\bar j}+b^iu_i+cu=f+\p_ih^i\text{ in } M=X\setminus D.
\end{equation}
Here $g^{i\bar j}$ is the inverse matrix of a K\"ahler cone metric $\om$ in $C^{,\a,\b}$.
We also denote the vector field $h^i\p_i$ to be $\mathbf h$ and $b^i\p_i$ to be $\mathbf b$.
Moreover, we are given the following data.
\begin{align}\label{linear equ ge co}
\mathbf h\in C^{1,\a,\b} \qquad \text{ and }\qquad  \mathbf b,c, f\in C^{,\a,\b} \;.
\end{align}

\begin{prop}[\cites{MR3405866,MR2975584,MR3144178}]\label{linearestimate}
	Fix $\a$ with $0<\a<\frac{1}{b}-1$.
	Then there is a constant $C$ depending on $\b$, $n$, $\a$, $|\mathbf b|_{C^{,\a,\b}}$, $|c|_{C^{,\a,\b}}$
	such that for all the functions $f\in C^{,\a,\b}$ and $\mathbf h\in C^{1,\a,\b}$, we have Schauder estimate of the weak solution
	of equation \eqref{linear equ ge},
	$$
	|u|_{C^{2,\a,\b}}\leq C(\|u\|_{L^{\infty}}+|f|_{C^{,\a,\b}}+|\mathbf h|_{C^{1,\a,\b}})\;.
	$$
\end{prop}


\subsection{Some properties of cscK cone metrics}\label{recentprogress}
In this section we review some recent progress on the theory of cscK cone metrics and show some extra properties.

Recall the definition of the cscK cone metrics in \cite{zhengcscKcone}.
\begin{defn}\label{zhengcscKcone}
	We say that $\om_{cscK}$ is a cscK metric with conical singularities if
	\begin{itemize}
		\item
		$\om_{cscK}$ is a cscK metric on the regular part $M$;
		\item
		$\om_{cscK}$ is quasi isometric to the model metric $\om_{cone}$;
		\item the potential of $\om_{cscK}$ lies in $C^{2,\a,\b}$.
	\end{itemize}
\end{defn}

From the definition, the cscK cone metric satisfies the equation on the regular part $M$,
\begin{align}\label{4th equ}
S(\om_{cscK})=\underline S_\b.
\end{align}
\begin{rem}
	We only require the second order behavior of the cscK cone metric in this definition.
	There are different ways to define cscK metrics with conical singularities and different notions are compared in \cite{LiL}. However, a crucial issue is the question of higher regularities of such metrics.
	 
\end{rem}

We write the cscK cone metric $\om_{cscK}$ using $\om_D$-potentials i.e $$\om_{cscK}:=\om_D+i\p\bar\p \vphi_{cscK}.$$

Because $\vphi_{cscK}$ is $C^{2,\a,\b}$, the 4th order equation \eqref{4th equ} could be re-written as the couple system of two second order elliptic equations
\begin{equation}\label{2nd equ}
\left\{
\begin{aligned}
\frac{\om_{cscK}^n}{\om_D^n}&=e^P,\\
\tri_{cscK} P&=g_{cscK}^{i\bar j} R_{i\bar j}(\om_D)-\underline S_\b.
\end{aligned}
\right.
\end{equation}

The following higher regularity theorem is proved in \cite{LZ2}.
\begin{thm*}[\cite{LZ2}]\label{cscKconeregularity}
	Assume that $\vphi_{cscK}$ is the potential of a cscK cone metric satisfying \eqref{2nd equ} with $\vphi_{cscK}$ is $C^{2,\a,\b}(X)$. Assume that the angle $2\pi\beta$ and the H\"older exponent $\alpha$ satisfy Condition \eqref{anglerestriction}. Then $\vphi_{cscK}$ is actually in $C^{4,\a,\b}(X;\om_D)$ and the Ricci curvature of $\om_{cscK}$ is $C^{,\a,\b}(X)$.
\end{thm*}

Next lemma appeared in \cite{MR3313212} in the case of K\"ahler-Einstein cone metrics.

\begin{prop}\label{averagecsck}Suppose that $\om_{cscK}\in [\omega]$ is a cscK cone metric with $C^{2,\a,\b}$ potential under Condition \eqref{anglerestriction}. The average of the scalar average of the cscK cone metric is
	\begin{align}
	\underline S_\b=n\frac{2\pi c_1(X)\cup [\omega]^{n-1}}{[\omega]^n}-n(1-\beta)\frac{2\pi c_1(D)\cup [\omega]^{n-1}}{[\omega]^n}.
	\end{align}
\end{prop}

\begin{proof} Since $\om_{cscK}$ has $C^{2,\a,\b}$ potential, 
$\om_{cscK}^n=|s|_h^{2\beta-2}\om_0^n e^\psi$ for some $\psi$
in $C^{,\alpha,\b}$ and $h$ a smooth hermitian metric on $\mathcal{O}(D)$. Then we have \begin{align}\label{Riccicsck}
Ric(\om_{cscK})=Ric(\omega_0)-(1-\beta)\Theta(h)+2\pi(1-\beta)[D]-i\p\bar\p\psi,
\end{align}
as an equality of closed currents. 
By definition the scalar curvature is trace of Ricci curvature, and so
\begin{align*}
\int_{M}S(\om_{cscK})\frac{\om_{cscK}^n}{n!}=& \int_{M}Ric(\om_{cscK})\wedge\frac{\om_{cscK}^{n-1}}{(n-1)!},\\
=&
\int_{M}Ric(\omega_0)\wedge\frac{\om_{cscK}^{n-1}}{(n-1)!}-(1-\beta)\int_{M}\Theta(h)\wedge\frac{\om_{cscK}^{n-1}}{(n-1)!}\\
&-\int_{M}i\p\bar\p\psi\wedge\frac{\om_{cscK}^{n-1}}{(n-1)!}.
\end{align*}
The first two terms are what we need, i.e equal to 
\begin{align*}
\frac{2\pi c_1(X)\cup [\om_{cscK}]^{n-1}-2\pi(1-\beta) c_1(D)\cup [\om_{cscK}]^{n-1}}{(n-1)!}.
\end{align*}
We claim that the third term vanishes.
From regularity Theorem \ref{cscKconeregularity}, the potential of our cscK cone metric is actually $C^{4,\a,\b}$. Hence, we have
$Ric(\om_{cscK})\leq C_0\cdot \om_{cscK}$ for some constant $C_0>1$. Moreover, there is a constant $C_1>1$ such that $Ric(\omega_0)-(1-\beta)\Theta(h)\geq -C_1\cdot \om_{cscK}$ on $X$. We set $C_2=C_0+C_1$. Using now \eqref{Riccicsck},  we obtain $$C_2\cdot \om_{cscK}+i\p\bar \p\psi\geq0.$$ We notice that $\vphi_{cscK}$ and $\vphi_{cscK}+\frac{1}{C_2}\psi$ are both $\om_0$-psh functions which are globally bounded on $X$. We only need to check the vanishing along the direction $z^1$. According to the integration by part formula in \cite[Theorem 1.14]{MR2746347}, we conclude the claim.
\end{proof}

\begin{prop} 
 Consider $X$ a smooth Fano manifold. Suppose that $D$ is a smooth divisor which is $\mathbb{Q}$-linearly equivalent to $-\lambda K_X$, with $\lambda\in \mathbb{Q}^*_+$.  
Consider $\om=\om_{cscK}$ is a cscK cone metric in the class $2\pi c_1(X)$ along $D$ with angle $2\pi\beta$ and H\"older exponent $\alpha$. Assume that $(\alpha,\beta)$ satisfy Condition \eqref{anglerestriction}. Then $\om$ is actually a K\"ahler-Einstein cone metric satisfying the equation
 $$Ric(\om)=\nu \om + 2\pi(1-\beta)[D],$$ with $\nu=1-(1-\beta)\lambda$.
Conversely, such a K\"ahler-Einstein cone metric is also a cscK cone metric.
\end{prop}

\begin{proof} Using same notations as above, the function $z^1$ is the local defining function of the divisor $D$. The Poincar\'e-Lelong equation tells us that
$$2\pi[D]=\sqrt{-1}\partial\bar\partial \log |z^1|^2,$$
so the trace reduces to $g^{1\bar 1}\delta_{\{z^1=0\}}$. Note that $\om$ is $C^{,\a,\b}$ and quasi-isometric to the standard cone metric
$$\omega_0=|z^1|^{2\beta-2}\sqrt{-1}dz^1\wedge d\bar z^1+\sqrt{-1}\sum_{i=2}^n dz^i\wedge d\bar z^i.$$ Since $\delta$ function is a generalized function of order 0 (i.e. its action can be continuously extended to $C^0$ functions), this implies $g^{1\bar 1}\delta_{\{z^1=0\}}=0$. Consequently, we have 
\begin{align}\label{vanish}tr_\omega [D]=0.\end{align}
 From Equation \eqref{Riccicsck}, $Ric(\om)$ is actually a representative of $2\pi c_1(X)-2\pi (1-\beta)c_1(D)=\nu c_1(X)$.
By considering cohomology classes, we can find a smooth real valued function $f$ such that
\begin{align}\label{eqRic}Ric(\omega)-\nu\omega -2\pi(1-\beta)[D]=\sqrt{-1}\partial\bar\partial  f.\end{align}
Taking trace with respect to $\omega$ and using \eqref{vanish}, we have
$$\triangle_\omega f=S(\omega)-n\nu.$$
Now, using the cscK condition, we obtain $\triangle_\omega f=0$. Thus $f$ is constant and we can conclude using again \eqref{eqRic}.
\end{proof}

Consider $V_f$ a holomorphic vector field on $X$ with holomorphy potential $f\in C^{\infty}(X,\mathbb{C})$ i.e $\iota_{V_f^{1,0}}\omega = -\bar{\partial}f$. Given the K\"ahler class $[\omega]$ and the vector field $V_f$, one can define the Futaki invariant as
$$Fut_{[\omega]}(V_f)=\frac{1}{2\pi}\int_X f\left(S(\omega) -\underline{S}\right) \frac{\omega^n}{n!}$$
where $\underline{S}$ is the average of the scalar curvature of any K\"ahler form in the class $[\omega]$ and
the log-Futaki invariant for vector fields $V_f$ that are furthermore tangent to $D$ as
\begin{align*}
Fut_{D,\beta,{[\omega]}}(V_f)=&\frac{1}{2\pi}\int_X f\left(S(\omega) - n\frac{2\pi c_1(X)\cup [\omega]^{n-1}}{[\omega]^n}\right) \frac{\omega^n}{n!} \\
&- (1-\beta)\left(\int_D f \frac{\omega^{n-1}}{(n-1)!} - n\frac{c_1(D)\cup [\omega]^{n-1}}{[\omega]^n}\int_X f \frac{\omega^n}{n!}\right),\\
=&Fut_{[\omega]}(V_f) - (1-\beta)\left(\int_D f \frac{\omega^{n-1}}{(n-1)!} - \frac{\Vol_{[\omega]} (D)}{\Vol_{[\omega]} (X)}\int_X f \frac{\omega^n}{n!}\right).
\end{align*}
Both Futaki invariants depend only on the class $[\omega]$.
As pointed out in \cite{Hasfutaki}, the log-Futaki invariant is the differential-geometric interpretation of the algebraic log Donaldson-Futaki invariant that can be defined using test configurations, see \cite{MR2975584} and \cite{MR3403730}.

Next corollary is known for K\"ahler-Einstein cone metrics on Fano manifolds, see for instance \cite{MR3470713}, or in the smooth case for the classical Futaki invariant. Consider $\omega_{cscK}\in [\omega]$ a cscK cone metric. It satisfies globally in the sense of distributions $$S(\omega_{cscK})=\underline{S}_\beta + 2\pi(1-\beta) \tr_{\omega_{cscK}} [D].$$ Applying  Proposition \eqref{averagecsck}, we obtain the following result.
\begin{cor} \label{logFutvanish}
Under assumption \eqref{anglerestriction}, the log-Futaki invariant $Fut_{D,\beta,[\omega]}$ vanishes on K\"ahler classes $[\omega]$ which contain cscK cone metric with cone singularities along $D$ with cone angle $2\pi \beta$.
\end{cor}

\begin{rem}[Expansion close to the divisor] For any cone angle $0<\b<1$, general expansion formulas for K\"ahler-Einstein cone metrics  appear in \cite{YZ}. They come from the study of a singular Monge-Amp\`ere equation. We expect that similar expansion formulas hold for the cscK cone metrics.
\end{rem}
\begin{rem}[Uniqueness of cscK cone metrics] The study of uniqueness of cscK cone metrics has been initiated in \cites{zhengcscKcone,LZ,LZ2,LZ3}.
\end{rem}

\section{Lichnerowicz equations with conical singularities}\label{Section2}
The Lichnerowicz operator at a cscK metric $\om_{cscK}$ is defined on functions $u$,
\begin{align}\label{Lic}
{\mathbb{L}\mathrm{ic}}_{cscK}(u)=\tri_{cscK}^2 u +  u^{i\bar j}R_{i\bar j}(\om_{cscK}).
\end{align}
We remark that when a K\"ahler metric has constant scalar curvature, the first variation of the scalar curvature is given by the Lichnerowicz operator.

We say that a K\"ahler cone metric $\om$ has bounded 	Christoffel symbols of the connection, if for any $2\leq i,j,k\leq n$, the following items are bounded,
\begin{equation*}
\left\{
\begin{aligned}
&\Gamma^{i}_{jk}, \quad 
|z^1|^{1-\b}\Gamma^{i}_{j1},\quad 
|z^1|^{\b-1}\Gamma^{1}_{jk},\\
&\Gamma^{1}_{j1}, \quad 
|z^1|^{2-2\b}\Gamma^{i}_{11}, \quad
|z^1|^{1-\b}(\Gamma^{1}_{11}+\frac{1-\b}{z^1});
\end{aligned}
\right.
\end{equation*}

We say that a K\"ahler potential $\vphi$ is a {\it $C^{4,\a,\b}$ (or $C^{3,\b}$) cscK potential} if $\om=\om_D+i\p\bar\p\vphi$ is a cscK metric and also $\vphi\in C^{4,\a,\b}$ (or $C^{3,\b}$ respectively).

We also consider the following operator on functions $u$,
\begin{align}\label{Licgeneral-def}
{\mathbb{L}\mathrm{ic}}_{\om}(u)=\tri_{\om}^2 u + u^{i\bar j}R_{i\bar j}(\om),
\end{align}
under the assumption that the coefficient metric $\om=\om_D+i\p\bar\p\vphi$ has $C^{4,\a,\b}$ cscK potential $\vphi$.

We are going to solve the equation for $f\in C^{,\a,\b}$
\begin{align}\label{Licgeneral}
{{\mathbb{L}\mathrm{ic}}}_{\om}(u)=f,
\end{align}
with solution $u\in C^{4,\a,\b}$ and prove the Fredholm alternative of the Lichnerowicz operator \thmref{Fredholm}.

\subsection{Sobolev spaces for cone metrics}
Since the volume element of the reference cone metric $\om$ is an $L^p$ function (for some $p\geq 1$) with respect to the Euclidean metric, it  gives rise to a measure $\om^n$ on $M$. Thus, we can introduce the following Sobolev spaces with respect to $\om^n$. We shall use the following complete Banach spaces on the whole manifold $X$.
\begin{defn}[Sobolev spaces $W^{1,p,\b}(\om)$] 
	For a K\"ahler cone metric $\omega$,  the Sobolev spaces $W^{1,p,\b}(\om)$ for $p\geq 1$
	are defined with respect to the reference K\"ahler cone metric $\om$. The $W^{1,p,\b}(\om)$ norm is 
	\begin{align*}
	\Vert u \Vert_{W^{1,p,\b}(\om)}&=\left(\int_M \vert u \vert^p + |\nabla u|_\om^p\om^n\right)^{1/p}.
	\end{align*}
\end{defn}

\begin{defn}[Sobolev spaces $W^{2,p,\b}(\om)$]\label{Sobolevspace}
	For a K\"ahler cone metric $\omega$ with bounded Christoffel symbols of the connection, the Sobolev spaces $W^{2,p,\b}(\om)$ for $p\geq 1$ are defined to be the completion of the space of smooth functions with finite $W^{2,p,\b}(\om)$ norm which is the combination of $W^{1,p,\b}(\om)$ norm and $W^{2,p,\b}(\om)$ semi-norm. The $W^{2,p,\b}(\om)$ semi-norm is defined with respect to  the reference K\"ahler cone metric $\om$ with cone angle $\beta$,
	
	\begin{align}\label{wapbseminorm}
	[ u ]_{W^{2,p,\b}(\om)}
	=&\sum_{1\leq a,b\leq n}\Vert \p_a\p_{\bar b}u\Vert_{L^p(\om)}
	+\sum_{2\leq j\leq n}\Vert \p_1\p_j u\Vert_{L^p(\om)}\\
	&+\sum_{2\leq j,k \leq n}\Vert \p_j\p_k u\Vert_{L^p(\om)}\nonumber.
	\end{align}
\end{defn}

\begin{defn}[Strong Sobolev spaces $W_{\bf s}^{2,p,\b}(\om)$]\label{strongSobolevspace}
	For a K\"ahler cone metric $\omega$ with bounded Christoffel symbols of the connection, the Sobolev spaces $W_{\bf s}^{2,p,\b}(\om)$ for $p\geq 1$ are defined to be the completion of the space of smooth functions with finite $W^{1,p,\b}(\om)$ norm and $W_{\bf s}^{2,p,\b}(\om)$ semi-norm. The $W_{\bf s}^{2,p,\b}(\om)$ semi-norm is defined with respect to the reference K\"ahler cone metric $\om$,
	\begin{align}\label{swapbseminorm}
	[ u ]_{W_{\bf s}^{2,p,\b}(\om)}&=(\int_M |\p\bar\p u|_\om^p+ |\p\p u|_\om^p\om^n)^{\frac{1}{p}}.
	\end{align}
	In both \eqref{wapbseminorm} and \eqref{swapbseminorm}, the second order pure covariant derivatives mean, for any $1\leq a,b\leq n$,
	\begin{align}\label{nablanabla}
	\p_a \p_b u:=\nabla_a \nabla_b u
	&=\frac{\p^2 u}{\p z^a\p z^b}-\sum_{c=1}^n \Gamma_{ab}^c(\om) \frac{\p u}{\p z^c}.
	\end{align}
	The Christoffel symbols $ \Gamma_{ab}^c(\om)$ of the connection  satisfy the properties of Corollary \ref{cor4ab}. 
\end{defn}

\begin{rem}
	From \eqref{nablanabla}, we could see clearly that why the bounded Christoffel symbols of the connection of the background metric $\om$ are required in the global definitions of the higher order Sobolev spaces. 
\end{rem}

\begin{defn}
	We define the Sobolev space $H^{2,\b}:=W^{2,2,\b}(\om)$ and $H_{0}^{2,\b}=\{u\in H^{2,\b}\vert \int_M u \om^n=0\}$. The Sobolev norm remains the same. The strong spaces $H_{\bf s}^{2,\b}:=W_{\bf s}^{2,2,\b}(\om)$ and $H_{{\bf s},0}^{2,\b}$ are defined in a similar way.
\end{defn}

\begin{lem}\label{w1pbSobolev imbedding theorem}
	Assume that $u\in W^{1,p,\b}(\om)$. If $p< 2n$, then there exists a constant $C$ such that
	\begin{align*}
	\Vert u\Vert_{L^{q}(\om)}\leq C\Vert u\Vert_{W^{1,p,\b}(\om)},
	\end{align*}for any $q\leq \frac{2np}{2n-p}$.
\end{lem}
\begin{proof}
	We consider the function which is supported in a cone chart, we use the map $W$ defined in \eqref{wcoordiante} and the Sobolev inequality in Euclidean space to obtain the desired inequality. The general case follows from a partition of unity. 
\end{proof}

\begin{lem}[Sobolev embedding theorem]\label{Sobolev imbedding theorem}
	Assume that $u\in W_{\bf s}^{2,p,\b}(\om)$, $p<2n$ and $q\leq\frac{2np}{2n-p}$.
	There exists a constant $C$ independent of $u$ such that
	\begin{align}
	\Vert u\Vert_{W^{1,q,\b}(\om)}\leq C\Vert u\Vert_{W_{\bf s}^{2,p,\b}(\om)}.
	\end{align}
\end{lem}
\begin{proof}
	From the lemma above,
	\begin{align*}
	\Vert\nabla u\Vert_{L^{q}(\om)}\leq C(\Vert\nabla|\nabla u|\Vert_{L^{p}(\om)}+\Vert\nabla u\Vert_{L^{p}(\om)}).
	\end{align*}
	So the conclusion follows from applying classical Kato inequality that gives
	$
	\Vert\nabla|\nabla u|\Vert_{L^{p}(\om)}\leq \Vert\nabla\nabla u\Vert_{L^{p}(\om)}.
	$ The R.H.S is controlled by the term $\Vert u\Vert_{W_{\bf s}^{2,p,\b}(\om)}$.
\end{proof}

\begin{lem}[Kondrakov compactness theorem]\label{compactness}
	Assume that $u\in W_{\bf s}^{2,p,\b}(\om)$, $p<2n$ and $q<\frac{2np}{2n-p}$.
	The Sobolev embedding $$W_{\bf s}^{2,p,\b}(\om)\subset W^{1,q,\b}(\om)$$ is compact.
\end{lem}
\begin{proof}
	We cover the manifold $X$ by a finite number of coordinates charts $\{U_i,\psi;1\leq i\leq N\}$ and let $\rho_i$ be the smooth partition of unity subordinate to $\{U_i\}$. Let $f_m$ be a bounded sequence in $W_{\bf s}^{2,p,\b}(\om)$. In the charts which do not intersect with the divisor, we let $\tilde f_m=(\rho_i f_m) \circ \psi_i^{-1}$. While, in the cone chart $U$ among $\{U_i\}$ we let $\tilde f_m=(\rho_i f_m) \circ \psi_i^{-1}\circ W^{-1}$. Then we are able to pick Cauchy subsequence of $f_m$ in each charts for $i=1,2,...,N$ because of precompactness of $\tilde f_m$ in each $U_i$. \end{proof}

\begin{prop}[Interpolation inequality]
	\label{lpinterpolation}
	Suppose that $\eps>0$ and $1<p<\infty$. There exists a constant $C$ such that for all $u\in W_{\bf s}^{2,p,\b}$, we have
	\begin{align*}
	\Vert u\Vert_{W^{1,p,\b}(\om)}\leq \eps \Vert u\Vert_{W_{\bf s}^{2,p,\b}(\om)}+ C\Vert u\Vert_{L^{p}(\om)}.
	\end{align*}
\end{prop}
\begin{proof}
	It follows from Lemma \ref{compactness} by using standard contradiction argument. We assume that the conclusion fails and for each $C_i=i$, there exists a $u_i\in W_{\bf s}^{2,p,\b}$ with $\Vert u_i\Vert_{W_{\bf s}^{2,p,\b}(\om)}=1$ such that
	\begin{align*}
	\Vert u_i\Vert_{W^{1,p,\b}(\om)}> \eps + i 
	\cdot \Vert u_i\Vert_{L^{p}(\om)}.
	\end{align*}Thus
	$\Vert u_i\Vert_{W^{1,p,\b}(\om)}\geq\eps>0$
	and $\Vert u_i\Vert_{L^{p}(\om)}\rightarrow 0$, as $i\rightarrow +\infty$. On one hand, from Kondrakov compactness (Lemma \ref{compactness}), after taking a subsequence, $u_i$ converges to $u_\infty$ in $W^{1,p,\b}(\om)$ norm. Also $\Vert u_\infty\Vert_{W^{1,p,\b}(\om)}\geq\eps>0$. On the other hand, from Sobolev embedding (Lemma \ref{Sobolev imbedding theorem}), $\Vert u_i\Vert_{W^{1,p,\b}(\om)}$ is uniformly bounded and then $u_i$ converges to zero in $L^p(\om)$ as $i\rightarrow \infty$. Thus $u_\infty=0$, contradiction!

\end{proof}

\begin{rem}
	
	One may wonder whether we could replace the $W_{\bf s}^{2,p,\b}$ with the partial norm $W^{2,p,\b}$ in all Lemma \ref{Sobolev imbedding theorem}, Lemma \ref{compactness} and Proposition \ref{lpinterpolation}. But, after examining the proof of Lemma \ref{Sobolev imbedding theorem}, it is obvious that a different Kato inequality would be needed.
	
\end{rem}

\begin{lem}[Poincar\'e inequality]\label{Poincareinequality}
	There is a constant $C_P$ such that for any $u\in H_{0}^{1,\b}=\{u\in H^{1,\b}\vert \int_M u \om^n=0\}$,
	\begin{align}
	\Vert u\Vert_{L^2(\om)} \leq C_P \Vert \nabla u\Vert_{L^2(\om)}. 
	\end{align}
\end{lem}
\begin{proof}
The Poincar\'e inequality follows from the compactness theorem i.e. the inclusion $W^{1,p,\b}(\om)\subset L^{q}(\om)$ with $q<\frac{2np}{2n-p}$ is compact. Actually, any minimizing sequence $u_i$ of $\Vert \nabla u\Vert_{L^2(\om)}$ over $H=\{u\in H_{0}^{1,\b} \text{ s.t }\Vert u\Vert_{L^2(\om)}=1\}$, converges strongly in $L^2$ and weakly in $H^{1,\b}_0$, to a limit $v$. So $\inf_{u\in H}\Vert \nabla u\Vert_{L^2(\om)}$ is realized by $v$ and has to be positive.	
\end{proof}

\subsection{A partial $L^p$ estimate}
In \cite{chenwang} (Definition 2.1), it is defined  a local Sobolev space $W_{loc}^{2,p,\b}(\mathtt{B};\om_{cone})$ over a ball $\mathtt{B}$ contained in a cone chart $U$. It contains the functions  $u\in W^{1,p,\b}(\mathtt{B};\om_{cone})$ such that for all $2\leq i,j\leq n$, 
\begin{itemize}
	\item 
	$|z^1|^{2(1-\b)}\frac{\p^2 u}{\p z^1\p z^{\bar 1}}\in L^p(\mathtt{B};\om_{cone})$;
	\item
	$|z^1|^{1-\b}\frac{\p^2 u}{\p z^1\p x^j}\in L^p(\mathtt{B};\om_{cone})$, for all $2\leq j\leq 2n$, with $z^i=x^i+\sqrt{-1}x^{n+i}$, for all $2\leq i\leq n$;
	\item
	$\frac{\p^2 u}{\p x^j \p x^k}\in L^p(\mathtt{B};\om_{cone})$, for all $2\leq j,k \leq 2n$;
	\item $u\in W^{2,2}(\mathtt{B}\setminus \mathtt{N}_\epsilon;\om)$ for any $\mathtt{N}_\epsilon\subset \mathtt{B}$, $\eps$-tubular neighbourhood of the divisor $D\cap \mathtt{B}$.
\end{itemize}
The semi-norm is defined to be
\begin{align}\label{wapbseminormloc}
[u]_{W^{2,p,\b}_{loc}(\mathtt{B},\om_{cone})}
=&\left\Vert |z^1|^{2(1-\b)}\frac{\p^2 u}{\p z^1\p z^{\bar 1}}\right\Vert_{L^p(\mathtt{B};\om_{cone})}\\
&+\sum_{2\leq j\leq 2n}\left\Vert|z^1|^{1-\b}\frac{\p^2 u}{\p z^1\p x^j}\right\Vert_{L^p(\mathtt{B};\om_{cone})} \nonumber \\
&+\sum_{2\leq j,k \leq 2n}\left\Vert\frac{\p^2 u}{\p x^j \p x^k}\right\Vert_{L^p(\mathtt{B};\om_{cone})}.\nonumber
\end{align}

It is also proved in \cite[Theorem 4.1]{chenwang}  a $L^p$ estimate over $U$ with respect to the flat metric $\om_{cone}$. Note that we define $\mathtt{B}_r$ the balls of radius $r$ times a small radii $r_0$ with respect to the cone metric $\om_{cone}$.

\begin{lem}[\cite{chenwang}, Theorem 4.1]\label{lppurelocal}
	Assume that $u\in W_{loc}^{2,p,\b}(\mathtt{B}_1;\om_{cone})$ for $2\leq p<\infty$, and $\tri_{\om_{cone}} u\in L^p(\mathtt{B}_2;\om_{cone})$. 
	Then there exists a constant $C$ depending on $n,p,\b$ such that
	\begin{align*}
	[u]_{W_{loc}^{2,p,\b}(\mathtt{B}_1;\om_{cone})}\leq C\cdot\Vert \tri_{\om_{cone}} u\Vert_{L^p(\mathtt{B}_2;\om_{cone})}.
	\end{align*}
	
\end{lem}
In order to extend the local definition to the global manifold, we need the following lemmas.
We could see that if we restrict the semi-norm $W^{2,p,\b}$ defined by \eqref{wapbseminorm} over $U$, it is controlled by patching up these local $W_{loc}^{2,p,\b}$ semi-norms on coverings.

\begin{lem}\label{lppure}Let $\om=\om_D+i\p\bar\p\vphi$ be a K\"ahler cone metric with bounded Christoffel symbols of the associated connection.
	There exists a constant $C>0$ such that for all function $u\in W_{loc}^{2,p,\b}(\mathtt{B}_1;\om_{cone})$, it holds
	\begin{align*}
	[ u ]_{W^{2,p,\b}(\mathtt{B}_1;\om)}\leq C\cdot [ u]_{W_{loc}^{2,p,\b}(\mathtt{B}_1;\om_{cone})}.
	\end{align*}
\end{lem}
\begin{proof}
	Recall the definition
	\begin{align*}
	[ u ]_{W^{2,p,\b}(\mathtt{B}_1;\om)}=&\sum_{1\leq a,b\leq n}\Vert \p_a\p_{\bar b}u\Vert_{L^p(\mathtt{B}_1;\om)}
	+\sum_{j=2}^n\Vert \p_1\p_j u\Vert_{L^p(\mathtt{B}_1;\om)}\\
	&+\sum_{2\leq j,k \leq n}\Vert \p_j\p_k u\Vert_{L^p(\mathtt{B}_1;\om)}.
	\end{align*}
	In the cone charts, $\om$ is equivalent to $\om_{cone}$. We then examine term by term. $\p_1\p_1 u$ is already in \eqref{wapbseminormloc}. Then $$\p_1\p_{\bar j}u=\frac{1}{2}(\frac{\p^2 u}{\p z^1\p x^j}-i \frac{\p^2 u}{\p z^1\p x^{n+j}}),$$ $\p_k\p_{\bar 1}$ and $\p_k\p_{\bar j}$ are also $L^p(\om_{cone})$.
	Meanwhile, the second term $$\p_1 \p_j  u=\frac{1}{2}(\frac{\p^2 u}{\p z^1\p x^j}-i \frac{\p^2 u}{\p z^1\p x^{n+j}})-\sum_{c=1}^n \Gamma_{1j}^c(\om) \frac{\p u}{\p z^c}.$$ From  assumption, the Christoffel symbols $\Gamma_{1j}^1(\om)$ and $|z^1|^{1-\b}\Gamma_{1j}^c(\om)$ for $2\leq c\leq n$ are all bounded. Therefore $\p_1 \p_j  u$ is $L^p(\om_{cone})$ and so is the third term $\p_j \p_k  u$.
\end{proof}
Then we consider the $W^{2,p,\b}(\om)$ solution of the linear equation on $M$,
\begin{align}\label{linerequation} 
\tri_\om u= f,
\end{align}
where $f\in L^p(\om)$.
\begin{prop}\label{lppureestimate}
	Let $\om$ be a K\"ahler cone metric with bounded Christoffel symbols of the associated connection. Suppose that $u\in W^{2,p,\b}(\om)$ for $2\leq p<\infty$ is a classical solution of  Equation \eqref{linerequation} for $f\in L^p(\om)$.
	Then there exists a constant $C$ depending on $n,p,\b,M,\om$ such that
	\begin{align*}
	\Vert u\Vert_{W^{2,p,\b}(\om)}\leq C(\Vert f\Vert_{L^p(\om)}+\Vert u\Vert_{W^{1,p,\b}(\om)}).
	\end{align*}
\end{prop}
\begin{proof}
	We let the manifold $B$ be covered by a finite number of coordinates charts $\{U_i,\psi;1\leq i\leq N\}$. We let $\rho_i$ be the smooth partition of unity, subordinate to $\{U_i\}$ and  supported on $\mathtt{B}_1\subset \mathtt{B}_{3}\subset U_i$ for each $i$. We compute that
	\begin{align*}
	\tri_\om (\rho_i u)
	=\tri_\om \rho_i u
	+ \rho_i \tri_\om u
	+2(\p \rho_i,\p u)_{\om}.
	\end{align*}
	Thus we apply \lemref{lppurelocal} (over $\mathtt{B}_{1}$ and $\mathtt{B}_{2}$) to 
	\begin{align*}
	\tri_{\om_{cone}} (\rho_i u)=(\tri_{\om_{cone}} (\rho_i u)-\tri_{\om} (\rho_i u))+\tri_\om \rho_i u
	+ \rho_i \tri_\om u
	+2(\p \rho_i,\p u)_{\om}:=f
	\end{align*} to obtain,
	\begin{align*}
	[ \textstyle\sum_{i}\rho_i u]_{W_{loc}^{2,p,\b}(\mathtt{B}_1;\om_{cone})}
	&\leq C\Vert f \Vert_{L^{p}(\mathtt{B}_2;\om_{cone})}
	\end{align*}
	and put together all estimates over each $U_i$ with Lemma \ref{lppure},
	\begin{align*}
	\Vert u\Vert_{W^{2,p,\b}(\om)}&=\left\Vert \textstyle\sum_{i}\rho_i u\right\Vert_{W^{2,p,\b}(\om)}\\
	&\leq C(\left\Vert \textstyle\sum_{i}\rho_i u\right\Vert_{W^{1,p,\b}(\om)}+[\textstyle\sum_{i}\rho_i u]_{W_{loc}^{2,p,\b}(\om_{cone})})\\
	&\leq C(\left\Vert \textstyle\sum_{i}\rho_i u\right\Vert_{W^{1,p,\b}(\om)}
	+\Vert f\Vert_{L^{p}(\om_{cone})})\\
	&\leq C(\Vert u\Vert_{W^{1,p,\b}(\om)}+\Vert  f\Vert_{L^{p}(\om)}).
	\end{align*}
	Then the conclusion follows.

\end{proof}


\subsection{Weak solutions to bi-Laplacian equations: existence}

For a constant $K>0$, we are looking for weak solutions to the following $K$-bi-Laplacian equation 
in $H_{0}^{2,\b}(\om)$,
\begin{align}\label{KbiLaplacian equation}
\tri_{\om}^2 u -K\tri_\om u=f.
\end{align}
The positive constant $K$ will be determined later in this section.

We recall the definition of the semi-norm $$ [ u ]_{H^{2,\b}(\om)}=\sum_{1\leq a,b\leq n}\Vert \p_a\p_{\bar b}u\Vert_{L^2(\om)}+\sum_{2\leq j\leq n}\Vert \p_1\p_j u\Vert_{L^2(\om)}+\sum_{2\leq j,k \leq n}\Vert \p_j\p_k u\Vert_{L^2(\om)},$$
which does not involve the term $\Vert \p_1\p_1 u\Vert_{L^2(\om)}$.
We introduce the following bilinear form.
\begin{defn}Define the bilinear form on $H_{0}^{2,\b}(\om)$ given by
	\begin{align*}
	\mathcal{B}^K(u,\eta):=\int_M[ \tri_\om u \tri_\om\eta+K g^{i\bar j}u_i\eta_{\bar j}]\om^n
	\end{align*}
	for all $u,\eta\in H_{0}^{2,\b}(\om)$.
\end{defn}
Then we introduce the weak solution to \eqref{KbiLaplacian equation}
, whose leading coefficients are conical.
\begin{defn}[Weak solution]
	We say that $u$ is the $H_{0}^{2,\b}(\om)$-weak solution of the $K$-bi-Laplacian equation \eqref{KbiLaplacian equation}, if it satisfies the following identity for all $\eta\in H_{0}^{2,\b}(\om)$,
	\begin{align*}
	\mathcal{B}^K(u,\eta)=\int_M f\eta\om^n.
	\end{align*}
\end{defn}

\begin{lem}\label{coercive}
	The bilinear form $\mathcal{B}^K$ is bounded and coercive for $K>C_P+1$. Here $C_P$ is the Poincar\'e constant of $\om$.
\end{lem}
\begin{proof}
	For all $u,\eta\in H_{0}^{2,\b}$, the boundedness follows from using Cauchy-Schwarz inequality,
	\begin{align*}
	\mathcal{B}^K(u,\eta)
	&\leq \Vert u\Vert_{H^{2,\b}(\om)}\Vert \eta\Vert_{H^{2,\b}(\om)}
	+K \Vert u\Vert_{H^{1,\b}(\om)}\Vert \eta\Vert_{H^{1,\b}(\om)}.
	\end{align*}
	Then we prove coercivity of the bilinear form. We first use the partial $L^2$ estimate of the cone metrics for the standard linear operator $\tri_\om $ (Proposition \ref{lppureestimate}), i.e. that there exists a constant $C>0$ such that,
	\begin{align*}
	\Vert u\Vert ^2_{H^{2,\b}(\om)}
	\leq C \int_M(|\tri_\om u|^2+|\nabla u|_\om^2+|u|^2)\om^n.
	\end{align*}
	Then, using the definition of the bilinear form, the R.H.S above is
	\begin{align*}
	&C\left( \mathcal{B}^K(u,u)+\int_M[(1-K)|\nabla u|_\om^2+|u|^2]\om^n\right).
	\end{align*}
	We use Poincar\'e inequality (\lemref{Poincareinequality}) and denoting the Poincar\'e constant by $C_P$,
	\begin{align*}
	\Vert u\Vert ^2_{H^{2,\b}(\om)}\leq  C \cdot (\mathcal{B}^K(u,u)+(1-K+C_P)\Vert \nabla u\Vert_{L^2(\om)}^2).
	\end{align*}
	Thus choosing $K>C_P+1$, we have
	\begin{align*}
	\mathcal{B}^K(u,u)\geq \frac{1}{C}\Vert u\Vert ^2_{H^{2,\b}(\om)}.
	\end{align*}
\end{proof}

The next proposition proves the existence of weak solution.
\begin{prop}\label{weaksolutionexistence}
	Let $\om$ be a K\"ahler cone metric with bounded Christoffel symbols of the connection. Suppose that $K> C_P+1$ and $f$ is in the dual space $ (H_{0}^{2,\b}(\om))^\ast$.
	Then the $K$-bi-Laplacian equation \eqref{KbiLaplacian equation} has a unique weak solution $u\in H_0^{2,\b}(\om)$.
\end{prop}
\begin{proof}
	According to \lemref{coercive}, the bilinear form $\mathcal{B}^K$ is bounded and coercive. Then the Lax-Milgram theorem tells us that there is a unique weak solution $v\in H_{0}^{2,\b}$ to equation \eqref{KbiLaplacian equation}. 
	
\end{proof}

	We could define another 2nd Sobolev space $H^{2,\b}_{\mathbf{w} }$ with semi-norm $$[ u ]_{H^{2,\b}_{\mathbf{w} }(\om)}=\sum_{1\leq a,b\leq n}\Vert \p_a\p_{\bar b}u\Vert_{L^2(\om)}$$ and norm $$|| u ||_{H^{2,\b}_{\mathbf{w} }(\om)}=|| u ||_{H^{1,\b}(\om)}+[ u ]_{H^{2,\b}_{\mathbf{w} }(\om)}.$$ Then following the same argument as above, we get a ``very weak''  solution, that lies in $H^{2,\b}_{\mathbf{w} ,0}$, the space of functions of $H^{2,\b}_{\mathbf{w} }$ with vanishing integral assuming that $\om$ is merely a K\"ahler cone metric.
\begin{prop}\label{weakersolution}
	Assume that $\om$ is a K\"ahler cone metric. Suppose that $K> C_P+1$ and $f$ is in the dual space $ (H_{\mathbf{w},0}^{2,\b}(\om))^\ast$.
	Then the $K$-bi-Laplacian equation \eqref{KbiLaplacian equation} has a unique weak solution $u\in H_{\mathbf{w},0}^{2,\b}(\om)$.
\end{prop}

\subsection{Weak solutions to bi-Laplacian equations: regularity}\label{Weak solutions of bi-Laplacian equations: regularity}

According to the existence theorem (Proposition \ref{weaksolutionexistence}), we now already have a weak solution $$u\in H_{0}^{2,\b}(\om). $$ We shall see we can obtain  that the weak solution is actually $C^{4,\a,\b}$ if additionally we impose more regularity on $f$.

\begin{prop}[Schauder estimate]\label{Schauder}
	With previous notations with $(\alpha,\beta)$ satisfying Condition \eqref{anglerestriction} and $f\in C^{,\a,\b}$, the weak solution $u\in H^{2,\b}_{0}$ to  equation \eqref{KbiLaplacian equation} is actually $C^{4,\a,\b}$. Moreover, there exists a constant $C$ such that
	\begin{align*}
	|\tri_\om u|_{C^{2,\a,\b}}\leq C(\Vert u\Vert_{H_{0}^{2,\b}(\om)}+|f|_{C^{,\a,\b}}).
	\end{align*}
\end{prop}
\begin{proof}
	We rewrite \eqref{KbiLaplacian equation} as 
	\begin{align}\label{triv}
	(\tri_\om-K) \tri_\om u= f.
	\end{align}
	According to the Schauder estimate for second order equation (see Section \ref{Second order elliptic equations with conical singularities}), we have proved that $\tri_\om u$ is in $C^{2,\a,\b}$ from \eqref{triv}. 
	Then we could use the angle restriction to conclude that $u\in C^{4,\a,\b}$, according to Proposition 4.3 in \cite{LZ2}.

\end{proof}

	Furthermore, we actually could weaken the condition on $\om$ using Proposition \ref{weakersolution}, however we do not use it in this paper.
\begin{prop}[Schauder estimate]\label{Schauderweakomega}
	Let $\om=\om_D+i\p\bar\p\vphi$ be a K\"ahler cone metric with $\vphi\in C^{2,\a,\b}$, $f\in C^{,\a,\b}$ and the H\"older exponent satisfy $\a\b<1-\b$, the weak solution $u\in H^{2,\b}_{\mathbf{w},0}$ to  equation \eqref{KbiLaplacian equation} is actually $C^{2,\a,\b}$. Moreover, there exists a constant $C$ such that
	\begin{align*}
	|u|_{C^{2,\a,\b}}\leq C(\Vert u\Vert_{H_{\mathbf{w},0}^{2,\b}(\om)}+|f|_{C^{,\a,\b}}).
	\end{align*}
\end{prop}
\begin{proof}
	We use \eqref{triv} again, we only need $\om$ to be K\"ahler cone metric to conclude $\tri_\om u$ is in $C^{\a,\b}$. 
	Then, 2nd order linear elliptic theory \cite{MR2975584} tells us that $u\in C^{2,\a,\b}$.	
\end{proof}

\subsection{Fredholm alternative for Lichnerowicz operator}\label{Continuity method}

We now use continuity method in the H\"older spaces. We define the operator $${{\mathbb{L}\mathrm{ic}}}^K_{\omega}(u)={{\mathbb{L}\mathrm{ic}}}_{\om}(u)- K \tri_\om u$$ and then define the continuity path $L_t^K: C^{4,\a,\b}\rightarrow C^{,\a,\b}$ with $0\leq t\leq 1$,
\begin{align}\label{continuitypath}
L_t^Ku
&= t {{\mathbb{L}\mathrm{ic}}}^K_{\omega}(u) +(1-t)(\tri^2_\om u-K\tri_\om u)\nonumber\\
&=\tri^2_\om u+ t  u^{i\bar j}R_{i\bar j}(\om)-K\tri_\om u.
\end{align}
Multiplying the equation with $u$ and integrating over the manifold $M$, we obtain the bilinear form 
\begin{align*}
\mathcal{B}^K_t(u,u)=\int_M u L_t^K  u \om^n=\int_M[|\tri_\om u|^2+ t u^{i\bar j}R_{i\bar j}(\om)u +K|\nabla u|_\om^2]\om^n.
\end{align*}

We will need the following lemma. 
\begin{lem}\label{IbPRic}
	Assume that $u\in C^{2,\a,\b}$ and $Ric(\om)$ is bounded. Then it holds
	\begin{align*}
	\int_M u^{i\bar j}R_{i\bar j}(\om)u \om^n
	=-\int_M u^{i}R_{i\bar j}(\om)u^{\bar j} \om^n.
	\end{align*}
\end{lem}
\begin{proof}
	We apply the cutoff function $\chi_\eps$ which has been fully discussed in \cite{LZ}. Then the argument is essentially Lemma 4.10 in \cite{LZ2}. By dominated convergence theorem, we have
	\begin{align*}
	\lim_{\eps\rightarrow 0}\int_M u^{i\bar j}R_{i\bar j}(\om)u\chi_\eps \om^n
	=\int_M u^{i\bar j}R_{i\bar j}(\om)u \om^n.
	\end{align*}
	On the other hand, using $\nabla Ric=0$ on $M$,
	\begin{align*}
	\int_M u^{i\bar j}R_{i\bar j}(\om)u\chi_\eps \om^n
	&=-\int_M u^{i}R_{i\bar j}(\om)(u \chi_\eps)^{\bar j} \om^n\\
	&=-\int_M u^{i}R_{i\bar j}(\om)u^{\bar j} \chi_\eps \om^n
	-\int_M u^{i}R_{i\bar j}(\om)\chi_\eps^{\bar j} u \om^n
	\end{align*}
	The first term converges under the assumption on $u$ and $Ric(\om)$.
	The second term also converges, since for $2\leq i,j\leq n$,  
	\begin{align*}
	&u^{1}R_{1\bar 1}(\om)\chi_\eps^{\bar 1}=\eps. \mathrm{o}(\rho^{-\b}),\quad u^{1}R_{1\bar j}(\om)\chi_\eps^{\bar j}=\eps. \mathrm{o}(1),\\
	&u^{i}R_{i\bar 1}(\om)\chi_\eps^{\bar 1}=\eps. \mathrm{o}(\rho^{-\b}),\quad u^{i}R_{i\bar j}(\om)\chi_\eps^{\bar j}=\eps. \mathrm{o}(1).
	\end{align*}
\end{proof}

When $t=0$, $L_0 u=\tri^2_\om u-K\tri_\om u$. We could solve $L_0 u=f$ for any $f\in C^{,\a,\b}$ and obtain an solution $u\in C^{4,\a,\b}$ thanks to Propositions \ref{weaksolutionexistence} and  \ref{Schauder}. 

In order to apply the continuity method to our linear PDE (see e.g. Theorem 5.2 in \cite{MR1814364} for Banach space setting), we need to prove the following key estimate.
\begin{thm}\label{closedness}
Assume $\omega$ is a cscK cone metric with $C^{4,\a,\b}$ potential, $(\alpha,\beta)$ satisfy the condition \eqref{anglerestriction}. Assume that $K>0$ is large enough, i.e. $K>1+2 \Vert Ric(\om)\Vert_{L^\infty}+3C_P$.
There is constant $C_1$ such that for any $u\in C^{4,\a,\b}$ along the continuity path \eqref{continuitypath} with $0\leq t\leq 1$, we have
\begin{align*}
|u|_{C^{4,\a,\b}}\leq C_1 |L^K_t u|_{C^{,\a,\b}}.
\end{align*}
\end{thm}
\begin{proof}
Applying Proposition 4.3 in \cite{LZ2} to the equation
\begin{align*}
\tri^2_\om u=L_t^Ku- t  u^{i\bar j}R_{i\bar j}(\om)+K\tri_\om u,
\end{align*}
we have
\begin{align*}
|u|_{C^{4,\a,\b}}\leq C_2 \left(|L^K_t u-t  u^{i\bar j}R_{i\bar j}(\om)+K\tri_\om u|_{C^{,\a,\b}}+|u|_{C^{,\a,\b}}\right).
\end{align*}
Note that $Ric(\om)$ is $C^{,\a,\b}$ and we apply the $\eps$-interpolation inequality of the H\"older spaces to the 2nd and 3rd terms on the right hand side, and so
\begin{align}
|u|_{C^{4,\a,\b}}\leq C_3 \left(|L^K_t u|_{C^{,\a,\b}}+|u|_{C^{,\a,\b}}\right).
\end{align}
We use the $\eps$-interpolation inequality of the $L^p$ spaces to  the 2nd term on the right hand side (Proposition \ref{lpinterpolation}, choosing $\eps$ small enough),
\begin{align}\label{path4ab}
|u|_{C^{4,\a,\b}}\leq C_4 \left(|L^K_t u|_{C^{,\a,\b}}+||u||_{L^2(\om)}\right).
\end{align}
Since $u\in C^{4,\a,\b}$, we are able to apply the integration by parts to obtain a G\"arding inequality as following.
We first use the $L^2$ estimate of the cone metrics to the standard linear operator $\tri_\om u$ (Proposition \ref{lppureestimate}), i.e. there exists a constant $C_5>0$ such that,
\begin{align*}
\Vert u\Vert ^2_{H^{2,\b}(\om)}
\leq C_5 \int_M(|\tri_\om u|^2+|\nabla u|_\om^2+|u|^2)\om^n.
\end{align*}
Here the integration  by parts works since $u\in C^{4,\a,\b}$. 
Using the bilinear form $\mathcal{B}^K_t(u,u)$, the R.H.S of previous inequality is
\begin{align*}
&C_5\left(\mathcal{B}^K_t(u,u)+\int_M[- t u^{i\bar j}R_{i\bar j}(\om)u +(1-K)|\nabla u|_\om^2+|u|^2]\om^n\right).
\end{align*}
Then we use the Cauchy-Schwarz inequality to $\mathcal{B}^K_t(u,u)=\int_M u L_t^K  u \om^n$,
\begin{align*}
\Vert u\Vert ^2_{H^{2,\b}(\om)} \leq C_5&\left( \Vert L^K_t u\Vert_{L^2(\om)}^2\right.\\
                                 &\pushright{\qquad\left.+\int_M[- t u^{i\bar j}R_{i\bar j}(\om)u +(1-K)|\nabla u|_\om^2+2|u|^2]\om^n\} \right)}.
\end{align*}
Integrating by parts the term containing $Ric(\om)$ using \lemref{IbPRic}, the R.H.S of last inequality becomes
\begin{align*}
C_5\left(\Vert L^K_t u\Vert_{L^2(\om)}^2+\int_M[ t u^{i}R_{i\bar j}(\om)u^{\bar j} +(1-K)|\nabla u|_\om^2+2|u|^2]\om^n\right).
\end{align*}
Thus,
\begin{align*}
\Vert u\Vert ^2_{H^{2,\b}(\om)} \leq C_5 &\left( \Vert L^K_t u\Vert_{L^2(\om)}^2\right. \\
&\left.\qquad+\int_M[(1-K+\Vert Ric(\om)\Vert_{L^\infty})|\nabla u|_\om^2+2 |u|^2]\om^n\right).
\end{align*}
Then we apply the Poincar\'e inequality (Lemma \ref{Poincareinequality}) to the 3rd term again and set $K_0=1-K+\Vert Ric(\om)\Vert_{L^\infty}+2C_P$. We obtain
\begin{align}\label{pureL2}
\Vert u\Vert ^2_{H^{2,\b}(\om)} &\leq C_5\left( \Vert L^K_t u\Vert_{L^2(\om)}^2+K_0\Vert\nabla u\Vert_{L^2(\om)}^2\right).
\end{align}
Now we use the special form of the Lichnerowicz operator to estimate the term $\p\p u$ (since $u\in C^{4,\a,\b}$), i.e.
\begin{align*}
\int_M u {{\mathbb{L}\mathrm{ic}}}_{\omega}(u) \om^n=\int_M |\p\p u|_\om^2 \om^n.
\end{align*}
Thus we use \eqref{continuitypath}, integration by parts (\lemref{IbPRic}) and Cauchy-Schwarz inequality as before,
\begin{align*}
\int_M |\p\p u|_\om^2 \om^n&=\int_M u [L^K_t u+(1-t) u^{i\bar j}R_{i\bar j}(\om)+K\tri_\om u]\om^n\nonumber\\
&\leq  \int_M(|L^K_t u|^2+ |u|^2+\Vert Ric(\om)\Vert_{L^\infty} |\nabla u|_\om^2-K|\nabla u|_\om^2)\om^n.
\end{align*}
We apply the Poincar\'e inequality (Lemma \ref{Poincareinequality}) to the 2nd term again and set $K_1=C_P + \Vert Ric(\om)\Vert_{L^\infty}-K$,
\begin{align}\label{singL2}
\int_M |\p\p u|_\om^2  \leq   \int_M(|L^K_t u|^2+K_1|\nabla u|_\om^2)\om^n.
\end{align}
Thus we add \eqref{pureL2} and $C_5$ times \eqref{singL2} together and have that
\begin{align*}
\Vert u\Vert ^2_{H_{\bf s}^{2,\b}(\om)}\leq C_5\left( 2 \Vert L^K_t u\Vert_{L^2(\om)}^2+ (K_0+K_1)\Vert\nabla u\Vert_{L^2(\om)}^2\right).
\end{align*} 
We further choose $K_0+K_1<0$ i.e. $2K>1+2 \Vert Ric(\om)\Vert_{L^\infty}+3C_P$, then we have
\begin{align*}
\Vert u\Vert ^2_{H_{\bf s}^{2,\b}(\om)}\leq 2C_5\Vert L_t^K u\Vert_{L^2(\om)}^2.
\end{align*} 
Together with \eqref{path4ab}, this allows us to conclude that
\begin{align}
|u|_{C^{4,\a,\b}}\leq C_1 |L^K_t u|_{C^{,\a,\b}}.
\end{align}
\end{proof}

\begin{proof}[Proof of Theorem \ref{Fredholm}]
We have just solved $$L_1^Ku={{\mathbb{L}\mathrm{ic}}}^K_{\omega}(u)={{\mathbb{L}\mathrm{ic}}}_{\omega}(u) -K u=f$$  and seen that the inverse map $({{\mathbb{L}\mathrm{ic}}}^K_{\om})^{-1}: C^{4,\a,\b}\rightarrow C^{4,\a,\b}$ is compact. 
Now, we can solve \eqref{Liceqn}, i.e
\begin{align}
{{\mathbb{L}\mathrm{ic}}}_{\omega}(u)={{\mathbb{L}\mathrm{ic}}}_\om^{K}(u)+K u=f.
\end{align}
Actually, this is equivalent, after taking $({{\mathbb{L}\mathrm{ic}}}^K_{\om})^{-1}$, to
\begin{align}
u+K  ({{\mathbb{L}\mathrm{ic}}}_\om^{K})^{-1} u= ({{\mathbb{L}\mathrm{ic}}}_\om^{K})^{-1}f.
\end{align}
Since $\mathfrak{T}:=-K  ({{\mathbb{L}\mathrm{ic}}}^K_{\om})^{-1}: C^{4,\a,\b}\rightarrow C^{4,\a,\b}$ is compact, we can apply classical results of functional analysis and Riesz-Schauder theory (see \cite[Theorem 5.3]{MR1814364}) to the Lichnerowicz operator which is self-adjoint. Furthermore the reductivity of the automorphisms group of a K\"ahler manifold admitting cscK cone metric with $C^{4,\a,\b}$ potential are proved in \cite{LZ2}, showing the one-one correspondence between the
kernel of the Lichnerowicz operator and the holomorphic vector fields tangential to the divisor.
\end{proof}

\section{Hermitian-Einstein metrics with conical singularities}\label{Section 3}


\subsection{Stable parabolic structures}\label{parabolicstab}
From now, we consider $E\to B$ a holomorphic vector bundle over a base $B$, compact K\"ahler manifold endowed with a smooth K\"ahler metric $\omega_0$.
Let $D=\sum^m_{i=1} D_i $ be a simple normal crossings divisor of $B$. 

\begin{defn}\label{defparab}
 A parabolic structure on $E$ with respect to $D$ consists of:
\begin{itemize}
 \item a filtration of $E_{\vert D_i}$ for $1\leq i \leq m$ such that
$$E_{\vert D_i} = \mathcal{F}_i^1 \varsupsetneq .. \varsupsetneq \mathcal{F}_i^{l_i}\varsupsetneq \{0\}$$
with $\mathcal{F}_i^{p+1}$ proper subbundle of $\mathcal{F}_i^p$ and the flags satisfy a natural compatibility condition: for every $I=(i_1,...,i_q)$, the restrictions $\{{\mathcal{F}_{i_l}^p}_{\vert D_{i_1}...D_{i_q}} , 1\leq l\leq q, 1\leq p \leq l_{i_l}\}$ to
$D_{i_1}...D_{i_q}$ yield to a flag of $E_{\vert D_{i_1}...D_{i_q}}$ which is a refined flag of $\{{\mathcal{F}_{i_l}^p}_{\vert D_{i_1}...D_{i_q}}, 1\leq p \leq l_{i_l}\}$ for every  $1\leq l\leq q$.
\item some real weights $\alpha_i^1, ... , \alpha_i^{l_i}$ attached to $\mathcal{F}_i^p$, $1\leq p \leq l_i$ satisfying the inequalities $0\leq \alpha_i^1<... < \alpha_i^{l_i}<1$.
\end{itemize}
\end{defn}

We recall a classical definition.
\begin{defn}
 Given $E$ a parabolic structure, one can define its parabolic degree with respect to $\omega_0$ as
$$\mathrm{par}\deg(E)=\deg(E)+ \sum_{i=1}^m \sum_{p=1}^{l_i} \rk(\mathcal{F}_i^p/\mathcal{F}_i^{p+1}) \alpha_i^p \deg(D_i)$$
and its parabolic slope as $\mathrm{par}\mu(E)=\mathrm{par}\deg(E)/\rk(E)$. Here the degree $\deg(E)$ is computed in the usual sense using the K\"ahler {cone} metric $\omega_0$ and depends only on the K\"ahler class. This definition extends to coherent subsheaves endowed with parabolic structures.
\end{defn}

 There exists a notion of stability for parabolic structures modeled on the notion of Mumford-Takemoto stability.
\begin{defn}
 Given a proper coherent subsheaf $F$ of a parabolic vector bundle $E$ along $D$, one can consider the induced parabolic structure for $F$. The only difficulty is to choose correctly the weights $\alpha_i^p(F)$. This is done by taking the maximum weights among the $\alpha_i^{p'}(E)$ that respect the flag structure, i.e $\mathcal{F}_{i}^p(F)\subset \mathcal{F}_{i}^{p'}(E)$. We say that $E$ is \textbf{parabolic stable} if for all proper coherent subsheaf $F$ of $E$, we have
$$\mathrm{par}\mu(F)< \mathrm{par}\mu(E).$$
\end{defn}
\subsection{H\"older spaces for bundle endomorphisms}\label{holderspacebundle}

Let $V$ be a holomorphic vector bundle over the base manifold $B$. Let us fix a holomorphic frame $F_V=\{e_\varsigma;1\leq \varsigma\leq \rk(V)\}$ and a finite covering $(U_i)_{i=0,..,N}$ of $B$ composed of local cone charts around $D$. Consider a partition of unity $\{\psi_i\}_{i=0,..,N}$ where $\psi_i \in C^{\infty}(U_i)$ have compact support,  associated to the finite covering $(U_i)_{i=0,..,N}$. Using the notations of subsection \ref{holderspace}, 
we define the space $C^{,\a,\b}(V)$ to be the space of sections $s$ of $V$ such that in $F_V$ the decomposition of $s$ is given by $\rk(V)$ functions $s_1,..,s_{\rk(V)}$ that lie in the space $C^{,\a,\b}$. More precisely, $s$ is a section in the H\"older space $C^{,\a}(V)$ such that if the local frame is defined close to the divisor $D$, over a cone chart $U_i$, each $\rk(V)$ complex valued functions defining $s$ lie in the space $C^{,\a,\b}(U_i)$. 
The advantage of fixing a frame and a partition of unity is that we can define now a norm $\Vert . \Vert_{C^{,\a,\b}}$ by $$\Vert s \Vert_{C^{,\a,\b}}= \sum_{i=0}^N \sum_{j=1}^{\rk(V)}\Vert \psi_i s_j \Vert_{C^{,\a,\b}(U_i)}.$$
Note that the space of sections with bounded $\Vert . \Vert_{C^{,\a,\b}}$ norm is independent of the covering, the partition of unity and the holomorphic frame.

Similarly to subsection \ref{holderspace}, we can define the vector spaces $C^{2,\a,\b}(V)$, $C^{3,\a,\b}(V)$, $C^{4,\a,\b}(V)$ and the associated norms by considering the analogue conditions on the decomposition of $s$. Eventually, all the spaces equipped with their natural norms are Banach spaces. 
With the previous reasoning, we can define this way the H\"older spaces of $End(E)$ that are denoted $C^{k,\a,\b}(End(E))$. 

The definition also applies to the space $\mathcal {H}^+(E)  $ of hermitian metrics on the bundle $E$ seen as sections of the frame bundle and for $Herm(E,h)$ the space of hermitian endomorphisms of $E$ (over $M\setminus D$) with respect to the hermitian metric $h$.

\subsection{Existence of Hermitian-Einstein cone metrics}\label{coneHE}
Given $\omega_0$ and the divisor $D$, we can consider a model metric $\omega_D$ as in Section \ref{Model cone metric}. Moreover, we will denote $\omega_B$ a K\"ahler cone metric in the same class and $C^{2,\a,\b}$ potential, which is quasi-isometric to $\omega_D$.

We introduce now the Hermitian-Einstein equation and express it in coordinates. Let $H=\{H_{\varsigma\bar{\tau}}\}$ be the hermitian matrix induced by $h$ in a local holomorphic frame $\{e_\varsigma;1\leq \varsigma\leq r\}$ for $E$ of rank $r$. 

\begin{defn}
We say a hermitian metric $h$ is $C^{k,\a,\b}$ for $k\in \mathbb Z^+=\{0,1,2,3,\cdots\}$, if the associated hermitian matrix $H=\{H_{\varsigma\bar{\tau}}\}$ is $C^{k,\a,\b}$, i.e. all its components $H_{\varsigma\bar{\tau}}$, $1\leq \varsigma,\tau\leq r$ are all $C^{k,\a,\b}$. 
\end{defn}

The curvature is given as a 2 form $$\frac{\sqrt{-1}}{2\pi}{F_{\varsigma k \bar  j}^{\tau}}dz^k \wedge d\bar{z}^j=\frac{\sqrt{-1}}{2\pi}\sum_{j,k} F_{k,\bar  j}dz^k\wedge d\bar{z}^j$$ with explicitly  
 \begin{align*}{F_{\varsigma k \bar  j}^{\tau}}=-\partial_{\bar j}(H^{\tau\bar{\gamma}}\partial_k H_{{\gamma} \bar{\varsigma}}).
 \end{align*}
 Lowing down the index, we have
  \begin{align*}
  {F}_{\varsigma\bar\tau k\bar j}
 &=\sum H_{\gamma \tau} {F_{\varsigma k \bar  j}^{\gamma}}\\
  &=-\p_{\bar j}\p_k H_{\varsigma\bar\tau}+\sum_{\gamma,\nu=1}^r H^{\gamma\nu}\p_k H_{\varsigma\bar\nu} \p_{\bar j} H_{\gamma\bar\tau}.
 \end{align*}
Fix  $\omega_B=\frac{\sqrt{-1}}{2}(g_B)_{k \bar j}dz^k \wedge d\bar{z}^{j}$ a K\"ahler metric on the base $B$. The Hermitian-Einstein equation $$\frac{\sqrt{-1}}{2\pi}\Lambda_{\omega_B} F_h= Cst \times Id_E$$ reads in coordinates 
   \begin{align*}
   Const&=g_B^{k\bar j}F_{\varsigma\bar\tau k \bar  j}\\ 
   &=-\tri_{\om_B} H_{\varsigma\bar\tau}+g_B^{k\bar j}\sum_{\gamma,\nu=1}^r H^{\gamma\bar\nu}\p_k H_{\varsigma\bar\nu} \p_{\bar j} H_{\gamma\bar\tau}.
    \end{align*} 
 
 \begin{defn}[Compatible metric with respect to parabolic structure]
  Let $h$ be a metric on the parabolic bundle $E$. We say that $h$ is compatible with the parabolic structure if the following holds. Given the parabolic structure, it is constructed by Li in \cite{LiJ} a model metric $h_{0}$ on $E$ over $B\setminus D$, such that
\begin{align}\label{modelmetric}\vert \Lambda_{\omega_D} F_{h_{0}}\vert_{h_{0}} \in L^{\infty}(B\setminus D),\quad \quad \vert  F_{h_{0}}\vert_{h_{0}} \in L^{p}(B\setminus D), p>1.\end{align} 
 This model metric on the bundle is natural. In a nutshell, the norm of a local section of $\mathcal{F}_i$ with respect to $h_0$ restricted to $D$ has growth controlled by the weights of the filtration. To be compatible for $h$ metric on $E$ means that the 2 following conditions hold:
 \begin{itemize}
  \item $h,h_0$ are mutually bounded;
  \item $\vert \bar{\partial} (h_0^{-1}h)\vert_{h_0} \in L^2(B,\omega_D)$. 
  \end{itemize}
 \end{defn}

 Given $\omega_D$ the model metric and $h$ a metric on $E$, compatible with respect to the parabolic structure, it is possible to compute the analytic degree of $E$. It is given by the differential geometry as
 $$\widetilde{\deg}(E)=\int_{B\setminus D} \tr\left(\frac{\sqrt{-1}}{2\pi}\Lambda_{\omega_D}F_h\right)\frac{\omega_D^n}{n!},$$
 and a similar formula applies for the proper coherent subsheaves of $E$. In \cite{LiJ} it is checked that $\widetilde{\deg}(E)$ is actually proportional to the parabolic degree of $\mathrm{par}\deg(E)$ and is an invariant of the space of hermitian metrics compatible with the parabolic structure. In other words, the bundle $E$ is parabolic stable if and only if it is stable with respect to the notion of slope induced by the analytic degree. \\
 The same property holds if we replace $\omega_D$ by $\omega_B$ as we assumed it has $C^{2,\a,\b}$ potential, and we have $\tilde{\mu}(E)=c \times \mathrm{par}\deg(E)$ for a certain constant $c>0$. The property of compatibility can also be defined using the K\"ahler cone metric $\omega_B$. In conclusion, we can speak of {\it parabolic stability of the parabolic bundle $E$ with respect to $\omega_B$} by using the analytic degree. 
 
 \medskip
 
We are ready to present a theorem of C. Simpson improved by J. Li.

\begin{thm*}[\,{\cite[Theorem 6.3]{LiJ}}, \cite{Sim}]\label{Simpson-Li}
Let $B$ be a base compact K\"ahler manifold endowed with a K\"ahler metric $\omega_B$ with conical singularities  along $D\subset B$, smooth divisor. Let $E$ a parabolic stable vector bundle over $B$ with respect to $\omega_B$. There exists $\delta_0>0$ such that if the angle $2\pi\beta$ of $\omega_B$ satisfies $0<\beta_i\leq\delta_0$, then there exists a Hermitian-Einstein metric $h_E$ on $E$ compatible with the parabolic structure over $D$. It
satisfies outside $D$ the Hermitian-Einstein equation, \begin{equation}
\frac{\sqrt{-1}}{2\pi}\Lambda_{\omega_B}F_{h_E}= \frac{\tilde{\mu}(E)}{\Vol} Id_E. \label{HEeqn}
\end{equation}
Here $Id_E$ is the identity endomorphism of $E_{\vert B\setminus D}$ and $\Vol$ the total volume of $B$ with respect to $\omega_B$.
\end{thm*}

We introduce the following definition of Hermitian-Einstein cone metric. 
 
 \begin{defn}\label{HEc-def}
 	As above, let $B$ be a base compact K\"ahler manifold endowed with a K\"ahler metric $\omega_B$ with conical singularities  along $D$, smooth divisor. Let $E$ a parabolic vector bundle with respect to  $D$ and $h_E$ a hermitian metric on $E_{\vert B\setminus D}$. We say that $h_E$ is a \textit{Hermitian-Einstein cone metric}, if $h_E$ satisfies the  Hermitian-Einstein equation \eqref{HEeqn} pointwisely over $B\setminus D$, $h_E$ is compatible with the parabolic structure and $h_E$ lies in  $C^{2,\a,\b}(\mathcal {H}^+(E))$.
 \end{defn}

\begin{thm}\label{HEsol}
Under same assumptions as in Theorem \ref{Simpson-Li}, the Hermitian-Einstein metric $h_E$ is actually a Hermitian-Einstein cone metric in the sense of Definition \ref{HEc-def}. \\
Moreover, if  $\alpha$ and $\beta$ satisfy the Condition (\ref{anglerestriction}), then $h_E\in C^{4,\a,\b}(\mathcal {H}^+(E)  )$.
\end{thm}
\begin{rem}\label{remHE}
 Note that the converse is true and constitutes the easy sense of the correspondence: an indecomposable parabolic vector bundle equipped with a Hermitian-Einstein cone metric compatible with its parabolic structure is actually parabolic stable. We refer \cite[Theorem 6.3]{LiJ}, \cite[Proposition 3.3]{Sim}.
\end{rem}

\begin{proof}
We start the proof by noticing that we could take a partition of unity $\{\theta_p\}$ of the base manifold $B$ subordinate to an open cover $\{U_p\}$ and construct the Hermitian-Einstein metric on each trivialization of the holomorphic vector bundle $E$ over each cover. It suffices to consider the cone chart $U$, which intersects with the divisor $D$, since far from the divisor all the arguments are the same to \cites{MR1165874,MR765366,MR885784,MR861491,MR1215276}.

 The proof is divided in several steps. We first use Dirichlet problem for Donaldson's flow to produce the weak solution $H$ to the Hermitian-Einstein equation away from the divisor and then prove the regularity of the  Hermitian-Einstein limit metric. Of course, the new improvement with our theorem is the regularity of the weak Hermitian-Einstein metric. Note that we do not try to improve the regularity of the flow itself, which is a parabolic system and the Schauder estimate is not yet known in this case. But instead, we use the limit equation and improve the regularity by observing that the nonlinear term itself is H\"older. As a result, each equation in the system is independent, and we are able to apply the elliptic regularity theorem for second order equations with conical singularities to each single equation of the system. 
 
 Using $h_0$ the model metric fixed by Li and which satisfies \eqref{modelmetric}, we write the endomorphism $$H=h {h_{0}}^{-1}.$$ We let $s$ be a section of $L_D$, $h_L$ be a smooth positively curved hermitian metric on the line bundle $L_D$. We denote by $D_\delta$ a $\delta$ tubular neighbourhood of the divisor $D$ for small $0<\delta\leq 1$. We also use $$\lambda=\frac{\tilde{\mu}(E)}{\Vol}.$$ Eventually, we omit the factor $\frac{\sqrt{-1}}{2\pi}$ in from of the contraction operator $\Lambda_\omega$ to ease notations. 
 
According to Donaldson \cite{MR1165874},
the Dirichlet problem for the following flow $h=h_t$,
\begin{equation}\label{deltaflow}
  \left\{
   \begin{aligned}
\dot hh^{-1}&=-( \L_{\omega_B} F_h-\lambda Id_E) \text{ over } U_\delta=U\setminus D_\delta,\\
h(x,0)&=h_0, \quad x\in U_\delta\\
h(x,t)&=h_0,\quad x\in \p U_\delta,  t\geq 0 \\
   \end{aligned}
  \right.
\end{equation}
has a unique global solution for $0\leq t< +\infty$.
We keep in mind that, we now obtain a sequence of solutions to Donaldson's flow on $\delta$, and denote the solutions by $h_\delta$. As further shown in \cite{MR1165874}, the convergence of \eqref{deltaflow} is irrelevant of the delicate conditions of stability. 
We will need that the approximation flow $h_\delta$ converges to a limit flow as $\delta\rightarrow 0$, the limit flow converges to a Hermitian-Einstein metric $h_\infty$ with conical singularities as $t\rightarrow+\infty$ and this limit metric $h_\infty$ has higher order regularity across the divisor $D$.
In order to achieves these goals, we need the following a priori estimates.

\textbf{Step: Uniform bound of $A(x,t):=|\L_{\omega_B} F_{h}|_h$.}
We have along the flow \eqref{deltaflow},
\begin{equation}\label{deltaflow2}
  \left\{
   \begin{aligned}
(\p_t-\tri_{\om_B})A(x,t)&\leq 0 \text{ in } U_\delta,\\
 A(x,0)&= |\L_{\omega_B} F_{h_0}|_{h_0}^2, \quad x\in U_\delta\\
 A(x,t)&= |\L_{\omega_B} F_{h_0}|_{h_0}^2,\quad x\in \p U_\delta,  t\geq 0.\\
   \end{aligned}
  \right.
\end{equation}
Let $A(t)=\sup_{U_\delta}A(x,t)$, then we apply the maximum principle, 
\begin{align}
\p_tA(t)\leq 0 .
\end{align}
So we prove that $\sup_{U_\delta}|\L_{\omega_B} F_h|_h$ is non-increasing along the flow, and also it is uniformly bounded by the initial given data  $|\L_{\omega_B} F_{h_0}|_{h_0}$ and independent of $\delta$ and $t$.

\textbf{Step: Zero order estimate.}
We are now aiming to prove that 
$h(t)$ converges to $h(T)$ in $C^0$ norm, for any finite time $T< \infty$. 
Donaldson's distance function between two Hermitian metrics is used,
\begin{align}
\sigma(h,k)=\tr h^{-1} k+\tr k^{-1} h -2 \rk(E).
\end{align}
It is known that for any two flows $h(t)$ and $k(t)$,
\begin{align}\label{decreasingdistancefunction}
\left(\frac{\p}{\p t}-\tri_{\om_B}\right)\sigma(h(t),k(t))\leq 0.
\end{align}
For any $\eps>0$, we choose $\kappa>0$ such that in the $\kappa$-neighbourhood of $t=0$, i.e.  for all $0\leq s,\tau<\kappa$, $$\sup_{U_\delta}\sigma(h_\delta(s),h_\delta(\tau))<\eps.$$
We now let $k_\delta(t)=h_\delta(t+s)$ in the inequality above, we see that $\sigma(h_\delta(t),k_\delta(t))$ is always zero on the boundary of the domain $U_\delta$ along the flow.
From maximum principle, we see that in the $\kappa$-neighbourhood of $T$, i.e. when $T-\kappa<s',\tau'<T$, $$\sup_{U_\delta}\sigma(h_\delta(s'),h_\delta(\tau'))\leq \eps.$$ Thus $h_t$ is a uniform Cauchy sequence and converges in $C^0$ norm to $h_T$.

\textbf{Step: Gradient estimate.}  It follows from the contradiction method, see  \cite[Lemma 6.4]{MR1040197}, that $h_t$ are bounded in $C^1$.

\textbf{Step: $W^{2,p}(\om_B)$ estimate.}   
From the gradient estimates above, $|H_{\varsigma\bar\tau}|_{C^1}$ is bounded and also is $\L_{\omega_B} F_h$.  So $-\tri_{\om_B} H_{\varsigma\bar\tau}$ is bounded by using the equation
\begin{align}\label{curvatureequ}
(\Lambda_{\omega_B} F_h)_{\varsigma\bar\tau}=-\tri_{\om_B} H_{\varsigma\bar\tau}+g_B^{k\bar j}\sum_{\gamma,\nu=1}^r H^{\gamma\bar\nu}\p_k H_{\varsigma\bar\nu} \p_{\bar j} H_{\gamma\bar\tau}.
\end{align} After applying the interior $L^p$ theory of the linear equation in $U_\delta$, we have for any $\varsigma,\tau$, $H_{\varsigma\bar\tau}\in W_{\bf s}^{2,p}(\om_B)(K)$ for any $K\subset\subset U_\delta$. And Proposition \ref{lppureestimate} for weaker Sobolev spaces tells us $H_{\varsigma\bar\tau}\in W^{2,p}(\om_B)(M)$. Thus $F_h$ is bounded in $L^p(M)$ norm for any $1\leq p<\infty$. Note that $W_{\bf s}^{2,p}$ is the strong Sobolev spaces and $W^{2,p}$ is the weaker one, see Definitions \ref{strongSobolevspace} and \ref{Sobolevspace}.

\textbf{Step: Long time existence.}
In each $U_\delta$, we can see deduce from the estimates above that the solution to the approximation equation \eqref{deltaflow} 
has long time existence. 
For any compact subset $K\subset\subset B\setminus D$, we could choose a small enough $\delta_K$ such that for any $\delta<\delta_K$, $K\subset\subset U_\delta$. 
Since the metrics $h(x,t)$ have uniform $W^{2,p}(\om_B)(M)$ estimates for any $p$, and independently of $\delta$, $h_\delta$ converges to a limit flow $h=\lim_{\delta\rightarrow 0} h_\delta$ in $W^{2,p}(\om_B)(M)$-norm for any $p\geq 1$ and furthermore $h$ solves
\begin{equation}\label{limitflow}
  \left\{
   \begin{aligned}
\dot hh^{-1}&=-(\L_{\omega_B} F_h-\lambda I) \text{ in } B\setminus D,\\
h(x,0)&=h_0, \quad x\in B\setminus D.
   \end{aligned}
  \right.
\end{equation}
 
\textbf{Step: Convergence.}
We need a subspace of the space of Hermitian metrics,
\begin{align*}
{\mathcal H}_{bounded}(E)=\Big\{ &h \text{ is a Hermitian metric on $E$ over $B\setminus D$ such that} \\
&\sup_{B\setminus D}|h|<+\infty\text{ and }\sup_{B\setminus D}\vert \Lambda_{\omega_B} F_{h}\vert_{h} < 2\sup_{B\setminus D}\vert\Lambda_{\omega_B} F_{h_{0}}\vert_{h_{0}} \Big\}.
\end{align*} 
Obviously, $h_0\in {\mathcal H}_{bounded}(E)$. Furthermore, for any $|h-h_0|_{C^{2,\a,\b}}\leq \eps$, we still have $ \sup_{B\setminus D}\vert\Lambda_{\omega_B} F_{h}\vert_{h} < 2\sup_{B\setminus D}\vert\Lambda_{\omega_B} F_{h_{0}}\vert_{h_{0}}$, provided that $\eps$ is small enough. Under the topology induced by the $C^{2,\a,\b}$-topology, we consider the path-connected branch of $h_0$, denoted by ${\mathcal H}_{bounded,h_0}(E)$, i.e. the set of metrics that be connected to $h_0$ by a path $$\{h_s,0\leq s\leq 1\}\subset {\mathcal H}_{bounded}(E).$$

Donaldson's functional for the cone version is well-defined on the space of Hermitian metrics ${\mathcal H}_{bounded,h_0}(E)$ with suitable asymptotic behavior near the divisor,
\begin{align*}
M_D(h_0,h)
=&\int_{0}^1 ds \int_B  \tr (\dot h_s h_s^{-1} \cdot F_{h_s}) \frac{ \omega_B^{n-1} }{(n-1)!}\\
&- \lambda\int_B \log \det(h_0 h^{-1})\frac{\omega_B^n}{n!},\\
=&\int_{0}^1 ds \int_B \tr (\dot h_s h_s^{-1} \cdot \L_{\omega_B} F_{h_s}) \frac{ \omega_B^{n} }{n!}\\
&- \lambda\int_B \log \det(h_0 h^{-1})\frac{\omega_B^n}{n!},
\end{align*}
where $h_s\in {\mathcal H}_{bounded,h_0}(E)$ is path connecting $h_0$ and $h$.
The definition is independent of the choice of the path. Actually, one can adapt to our setting the proof of the classical smooth case. The proof consists in showing that the variation of $M_D(h_0,.)$ is a closed 1-form. It requires to study the term $\phi_h:=\tr (h^{-1} \tilde{d}h\cdot F_{h})$ where $$h:\{(t,s),a\leq t\leq b,0\leq s\leq 1\}=\Delta\to {\mathcal H}_{bounded}(E)$$ is a smooth map and $\tilde{d}=(\partial_s)ds+(\partial_t)dt$ is the exterior differentiation on the domain $\Delta$. But the 1-form $\phi_h$ is well defined from our assumption on ${\mathcal H}_{bounded}(E)$ and one can apply Stokes theorem $\int_{\partial\Delta} \phi_h = \int_\Delta \tilde{d}\phi_h$. Then one can follow word by word  the proof of \cite[Lemma 3.6]{Kob}. Alternatively, one can show that the curvature of $h$ is a moment map for the action of the Gauge group on the space of Chern connections associated to ${\mathcal H}_{bounded}(E)$, see for instance \cite{MR2039989} and also \cite[Lemma 7.2]{Sim}. Then, as a classical result of the moment map theory, $M_D$ is the associated integral to this moment map and is consequently independent of the choice of the path.

We have by the arguments of \cite[Proposition 5.3]{Sim}, that there are two constants $C_1$ and $C_2$ such that for any $h\in {\mathcal H}_{bounded}(E)$
\begin{align}\label{logTr}
\Vert\log \Tr H\Vert_{L^1(\om_B)}^2\leq C_1+C_2 M_D(h_0,h).
\end{align}

Since $\L_{\omega_B} F_h$ and $h_0h^{-1}$ are both bounded, the following functional is  well-defined along the flow $h_t$,
\begin{align*}
M_D(h_0,h_t)
=&\int_{0}^t d\tau \int_B \tr (\dot h_\tau h_\tau^{-1} \cdot \L_{\omega_B} F_{h_\tau}) \frac{ \omega_B^{n} }{n!}\\
&- \lambda\int_B \log \det(h_0 h_t^{-1})\frac{\omega_B^n}{n!}.
\end{align*}

We need its first variation formula along the flow,
\begin{align}\label{monotonicity}
\frac{d}{dt}M_D(h_0,h_t)
=& \int_B \tr (\dot h_t h_t^{-1} \cdot \L_{\omega_B} F_{h_t}) \frac{ \omega_B^{n} }{n!}
- \lambda\int_B tr(\dot h_t h_t^{-1})\frac{\omega_B^n}{n!}, \nonumber\\
=&-\int_B |\L_{\omega_B} F_{h_t}-\lambda Id_E|^2 \frac{\omega_B^n}{n!}.
\end{align}
Thus the functional $M_D$ is non-increasing along the flow. This leads to a uniform upper bound to $\Vert\log \Tr H\Vert_{L^1(\om_B)}$ from \eqref{logTr}.

Now, we wish to apply De Giorgi-Nash-Moser iteration method for K\"ahler cone metrics of \cite[Section 4]{LZ} to the bounded $\log \Tr H$ in the following functional inequality
\begin{align}\label{logtrh}
\tri_{\om_B}\log \Tr H\geq -(|\L_{\omega_B} F_{h_0}|+|\L_{\omega_B} F_{h}|):=-f.
\end{align}
In order to do so, we need to examine the conditions of Proposition 4.8 in \cite{LZ}.
Firstly, the following Sobolev inequality with respect to $\om_B$ holds, i.e. for any $w \in W^{1,2}(\om_B)$, there is a Sobolev constant $C_S(\om_B)<+\infty$ such that
$$\|w\|^2_{L^{2^\ast}(\om_B)}\leq C_S(\om_B)(\|\nabla w\|^2_{L^2(\om_B)}+\|w\|^2_{L^2(\om_B)}),$$ 
where $2^\ast= \frac{2n}{n+1}$.
Secondly, we need to rewrite \eqref{logtrh} to the following form by using integration by parts, i.e. $v=\log \Tr H$ is a $W^{1,2} $ sub-solution of the linear equation in the weak sense, i.e. for any $\eta\in C^{2,\a,\b}$,
\begin{align}\label{linear equ ge2}
\int_B(\p v, \p\eta)_{\om_B} \om_B^n\leq -\int_Bf\eta\om_B^n.
\end{align} 
This is achieved by using the approximation sequence and $v$ vanishes on the exhaustion domains.
So, outside the $\delta$ neighbourhood of the divisor $D$,
\begin{align}
\int_{B\setminus D_\delta}(\p v, \p\eta)_{\om_{B}} \om_B^n
&=\int_{B\setminus D_\delta}-\tri v \eta  \om_B^n
+\int_{\p D_\delta} v \cdot \p\eta d\nu
&\leq -\int_Bf\eta\om_B^n.
\end{align}
The boundary $\int_{\p D_\delta} v \cdot \p\eta d\nu \rightarrow 0$, as $\delta\rightarrow 0$.

Let $$\tilde v=v-\frac{1}{V}\int_M v\,\om_B^n,$$ then there exists from \cite[Proposition 4.8]{LZ} a constant $C$ depending on the Sobolev constant $C_S$ with respect to $\om_B$ such that for $p^\ast=\frac{2np}{2n+p}$,
\begin{align}\label{glob bound claim}
\sup_B \tilde v \leq C(\| f \|_{L^{p^\ast}(\om_B)} +\| \tilde v\|_{L^1(\om_B)}).
 \end{align}
We thus obtain the $L^\infty$ bound of $H$,
\begin{align}
\sup_B |\log \Tr H|\leq C(\| f \|_{\infty} +\|\log \Tr H\|_{L^1(\om_B)}).
\end{align}

From the monotonicity of the energy along the flow \eqref{monotonicity}, the right hand side is uniformly bounded.
Letting $\eta=v$ in \eqref{linear equ ge2}, we have
\begin{align}
\Vert \log \Tr H\Vert_{W^{1,2}(\om_B)} \leq C.
\end{align} 
Thus we are able to prove that $H_t$ converges in $C^0$ norm to some $H_\infty$ and then get $C^1$ norm of $H_t$ which is independent of $t$, as the gradient estimate above in {\it Step: Gradient estimate}. After applying the $L^p$ theory of the linear equation with respect to K\"ahler cone metrics developed in \cite{chenwang} to the curvature equation \eqref{curvatureequ}, we have $H_{\varsigma\bar\tau}\in W^{2,p}(\om_B)$. Furthermore, the $W^{2,p}$ norm of $H(t)$ is independent of $t$, as the proof in {\it Step: $W^{2,p}(\om_B)$ estimate.}

\textbf{Step: $W^{1,2}(\om_B)$ weak solution.}
Now we have a Hermitian-Einstein metric on the regular part, but we still need to verify that the limit metric satisfies the Hermitian-Einstein equation in $W^{1,2}(\om_B)$ sense,
\begin{align}\label{HEequ}
Const\cdot I_{\varsigma\bar\tau}=-\tri_{\om_B} H_{\varsigma\bar\tau}+g_B^{k\bar j}\sum_{\gamma,\nu=1}^r H^{\gamma\bar\nu}\p_k H_{\varsigma\bar\nu} \p_{\bar j} H_{\gamma\bar\tau}.
\end{align} 
We fix $H= H_{\varsigma\bar\tau}$ and denote the nonlinear term $$\mathsf{N}=g_B^{k\bar j}\sum_{\gamma,\nu=1}^r H^{\gamma\bar\nu}\p_k H_{\varsigma\bar\nu} \p_{\bar j} H_{\gamma\bar\tau}.$$
We need
\begin{align}
C \int_B {\rm Id} \eta\om^n_B=\int_B (\p H,\p \eta)_{\om_B}\om^n_B
+\int_B \mathsf{N} \eta\om^n_B.
\end{align}
It suffices to use the approximation sequence again and prove the boundary term 
$
\int_{\p D_\delta} \p H\cdot \eta \rightarrow 0,$
as $\delta\rightarrow 0$. Thus is true, since $H$ has uniform $C^1$ norm.

\textbf{Step: $C^{4,\a,\b}$ estimate.}
Once we have $H_{\varsigma\bar\tau}$ is a $W^{1,2}(\om_B)$ weak solution and lies in $W^{2,p}(\om_B)$, we can apply the Sobolev embedding theorem \cite{chenwang} to obtain $H_{\varsigma\bar\tau}\in C^{1,\a,\b}$, thus returning to \eqref{HEequ}, the nonlinear term $\mathsf{N}$ is $C^{,\a,\b}$. Then we have $H_{\varsigma\bar\tau}\in C^{2,\a,\b}$ by Donaldson's Schauder estimate and bootstrap to $H_{\varsigma\bar\tau}\in C^{4,\a,\b}$ similar to the proof of Proposition \ref{Schauder} in Section \ref{Weak solutions of bi-Laplacian equations: regularity}.

\end{proof}

\begin{rem}
From \cite{LiJ}, one gets that $\delta_0$ depends on the (difference of the) weights of the parabolic structure of $E$. So a priori, $\delta_0$ is fixed and we don't know its size, it can be $>$ or $<1/2$. But we consider here the theorem only for angles $\beta< \min(1/2,\delta_0)$. In general, we believe the asymptotic behaviour of Hermitian-Einstein metric $h_E$ could be well-understood with the method in \cite{YZ} and our angle restriction could be removed.
\end{rem}
\subsection{Parabolic stability and holomorphic vector fields}
It is well-known that Mumford stable vector bundles are simple. In the parabolic setting, we have the following result.

\begin{lem}\label{lem1}
Assume $E$ is parabolic stable with parabolic structure along a simple normal crossings divisor $D=\sum^m_{i=1} D_i$.
Then the holomorphic endomorphism are the homotheties, i.e $$H^0(End(E))=\mathbb{C}.$$
\end{lem}
\begin{proof} The proof is similar to the non parabolic case. Let $f$ be a holomorphic endomorphism which is not zero or an isomorphism. Then by holomorphicity of $f$, $\ker(f), \mathrm{Im}(f)$ have a coherent subsheaf of $E$ with quotient torsion free. One can obtain a parabolic structure for $F=\ker(f), \mathrm{Im}(f)$ by intersection $F_{\vert D_i}$ with the elements of the flag of $E_{\vert D_i}$, discarding the subspaces of $F_{\vert D_i}$ that coincide with another one, and considering the associated largest parabolic weights. Thanks to the parabolic stability of $E$, we have now the inequalities
\begin{align}
 \frac{\mathrm{par}\deg(\ker(f))}{\rk(\ker(f))} <  \mathrm{par}\mu(E), \label{ieq1} \\
  \frac{\mathrm{par}\deg(\mathrm{Im}(f))}{\rk(\mathrm{Im}(f))}  <  \mathrm{par}\mu(E). \label{ieq2}
\end{align}
But, the parabolic weights of $\ker(f)$ and $\mathrm{Im}(f)$ satisfy also
\begin{equation}
 \mathrm{par}\mu(E) = \frac{\mathrm{par}\deg(\ker(f)) +\mathrm{par}\deg(\mathrm{Im}(f))  }{\rk(\ker(f)) + \rk(\mathrm{Im}(f))}. \label{ieq3}
\end{equation}
Using inequalities \eqref{ieq1}, \eqref{ieq2} and Equation \eqref{ieq3}, one gets  a contradiction: $\mathrm{par}\mu(\ker(f))<\mathrm{par}\mu(\mathrm{Im}(f))$ and  $\mathrm{par}\mu(\mathrm{Im}(f))<\mathrm{par}\mu(\ker(f))$. Thus, $f$ is an isomorphism or trivial. In the first case, fix $x\in B$ and consider any eigenvalue of $f: E_x \rightarrow E_x$. Let's call $\lambda_0$ this eigenvalue. Then by the reasoning as above, $f-\lambda_0 Id_E$ is zero and consequently
$f$ is an homothety.
\end{proof}

We need the following classical definition of logarithmic tangent bundle.
\begin{defn}
 Consider $B$ a complex manifold of complex dimension $n$ and $D$ a divisor with simple normal crossings singularities. In local coordinates $D=\{z \text{ s.t } \prod_{i=1}^d z^i=0\}$. Then the logarithmic tangent bundle $TB(-\log D)$ is the locally free sheaf generated by the vector fields
 $z^i \frac{\partial}{\partial z^i}$ where $1\leq i \leq d$ and the vector fields $\frac{\partial}{\partial z^i}$ for $d<i\leq n$.\\
Let us consider the stratification of $B$ given by $B_0=B\setminus D$, $B_1=D\setminus Sing(D)$ and recursively $B_k$ is the non-singular part of $Sing(B_{k-1})$. $T_B(-\log D)$ can be seen as the sheaf of holomorphic vector fields $v$ on $B$ such that for every $k\geq 0$, every $x\in B_k$, $v_x$ is tangent to $B_k$.
\end{defn}

 We denote the $X$ the projectivised bundle and $\pi$ the associated projection map to $B$,
 $$X:=\mathbb{P}E^*, \hspace{1cm} \pi:X\to B.$$
 Next, we derive some information on the holomorphic automorphisms of $X$ when $E$ is parabolic stable and $B$ has no nontrivial holomorphic vector field.

\begin{cor}\label{cornovectorfield}
 Assume $E$ is a parabolic stable vector bundle with respect to the K\"ahler cone metric  $\omega_B$ and the base $B$ has no nontrivial holomorphic vector field. \\
 Then the Lie algebra $Lie(Aut_{\mathcal{D}}(X,[k\pi^*\omega_B+\hat{\omega}_E]))$ is actually trivial.
\end{cor}
\begin{proof}
Consider $TFibre(-\log \mathcal{D})$ the sheaf of logarithmic tangent vectors to the fibre of $\pi$ with respect to $\mathcal{D}$.
There is an exact sequence
$$0\rightarrow TFibre(-\log \mathcal{D}) \rightarrow TX(-\log \mathcal{D}) \rightarrow \pi^*( TB(-\log D))\rightarrow 0$$
which provides a long exact sequence
\begin{align*}
0\rightarrow H^0(X,TFibre(-\log \mathcal{D})) \rightarrow H^0&(X,TX(-\log \mathcal{D})) \\
 & \rightarrow H^0(X,\pi^*( TB(-\log D)))\rightarrow ...
\end{align*}
Now, $H^0(X,\pi^*( TB(-\log D)))=0$ by assumption, and thus we obtain that the space $H^0(X,TFibre(-\log \mathcal{D}))$ is isomorphic to  $H^0(X,TX(-\log \mathcal{D}))$. \\
On another hand, $H^0(B,\pi_*TFibre(-\log \mathcal{D}))$ can be identified with the holomorphic parabolic endomorphisms along $D$ that are trace free. By parabolic endomorphisms, we mean endomorphisms of $E$ that preserve the parabolic structure of $E$. Hence,
$$H^0(X,TFibre(-\log \mathcal{D}))\simeq H^0(B,\pi_*TFibre(-\log \mathcal{D})).$$ Eventually, using Lemma \ref{lem1}, we obtain $H^0(X,TFibre(-\log D))=0$ and thus, $H^0(X,TX(-\log \mathcal{D}))=0$. This means that there is no nontrivial holomorphic vector fields tangent to $D$.
\end{proof}


\section{Construction of cscK cone metrics over projective bundles}\label{construction}
Given a (hermitian) vector space $\Xi$, there is an isomorphism  between $\Xi$ and $H^0(\mathbb{P}\Xi^*,\mathcal{O}_{\mathbb{P}\Xi^*}(1))$.
This leads to define a metric on $\mathcal{O}_{\mathbb{P}\Xi^*}(1)$ by the following construction. For $v\in \Xi$, the element $\hat{v}\in H^0(\mathbb{P}\Xi^*,\mathcal{O}_{\mathbb{P}\Xi^*}(1))$ is such that
$\hat{v}(\xi)=\xi(v)$ for $\xi \in \Xi^*$. Then from any metric $h$ on $\Xi$, we get a metric $h^*$ on $\Xi^*$ and a Fubini-Study metric $\hat{h}$ on the line bundle $\mathcal{O}_{\mathbb{P}\Xi^*}(1)$
by the formula \begin{equation}\label{hat}\hat{h}(\hat{v},\hat{w})(\xi)=\frac{\xi(v)\overline{\xi(w)}}{\vert\xi\vert^2_{h^*}}\end{equation}
for $v,w\in \Xi$, and $\xi\in \Xi^*$.\\
Consequently, from the Hermitian metric $h_E$ on the holomorphic vector bundle $E$, we get a Hermitian metric $\hat h_E$ on the line bundle $\mathcal{O}_{\mathbb{P}E^*}(1)$.
 
\subsection{Construction of background metrics}\label{background}

Over $X$, we consider the K\"ahler metric $\hat{\omega}_E \in c_1(\mathcal{O}_{\mathbb{P}E^*}(1))$ outside $\mathcal{D}$ given by the formula
\begin{align}
\hat{\omega}_E=i\bar\p\p\log \hat h_E.
\end{align}
\begin{prop}\label{wE} Let $h_E$ the Hermitian-Einstein cone metric obtained in Theorem \ref{HEsol}. Then 
$\hat{\omega}_E$ is a (1,1)-form on $X={\mathbb{P}E^*}$ with conical singularity along $\mathcal{D}$ in $C^{2,\a,\b}$ topology.
\end{prop}
\begin{proof}
The local computation of $\hat{\omega}_E$ using \eqref{hat} involves only terms of the form $\partial_k H_{\varsigma \bar{\tau}}$, $\partial_j\partial_{\bar k} H_{\varsigma \bar{\tau}}$
 where $H$ is the matrix representing $h$ in a local frame close to $D$. Thus, in order to have a cone metric we only need to have the entries of $H$ to be $C^{4,\a,\b}$, and $h_E \in C^{4,\a,\b}(\mathcal {H}^+(E)  )$, which is the case from Theorem \ref{HEsol}.
\end{proof}

Thus, from the metric on the base $B$, we obtain a K\"ahler cone metric in $[k\pi^*\omega_B+\hat{\omega}_E]$, that we denote by  \begin{align}\label{omegak}\omega_k=k\pi^*\omega_B+\hat{\omega}_E.\end{align}
\begin{lem} \label{regomegak}  Let $h_E$ the Hermitian-Einstein cone metric obtained in Theorem \ref{HEsol} and assume that condition \eqref{anglerestriction} holds.  Then we have 
$\omega_k\in C^{2,\a,\b}$ and $S(\om_k)\in C^{,\a,\b}$.
\end{lem}
\begin{proof}
This is an application of Theorem \ref{cscKconeregularity}. The cscK cone metric $\omega_B$ on the base is $C^{2,\a,\b}$, while $\hat{\omega}_E$ is also $C^{2,\a,\b}$ from the previous result.
\end{proof}


\subsection{Expansion of scalar curvature}
Let us remember that we know from Hong's techniques.
\begin{lem}\label{Hongt}
On the regular part of $\mathbb{P}E^*$,
\begin{align}\nonumber S(\omega_k)([v])=&r(r-1)+\frac{1}{k}\left(\pi^* S(\omega_B) + 2r\frac{\sqrt{-1}}{2\pi}\Lambda_{\omega_B} \mathrm{tr}\left([F_{h_E}]^0\frac{v\otimes v^{*_{h_E}}}{\Vert v\Vert^2}\right)\right)\\
&+ O\left(\frac{1}{k^2}\right) \label{eqS}
\end{align}
where $r=\rk(E)$, $[v]\in \mathbb{P}E^*$ and $[.]^0$ denotes the trace free part.
\end{lem}


\subsection{Approximate cscK cone metrics}
In view of Theorem \ref{thm1}, we shall deform the metric $\omega_k$ to obtain K\"ahler metrics $$\omega_{k,p}=\om_k+\sqrt{-1}\p\bar\p\phi_{k,p}, (p>0, k>>0)$$ as in the following proposition, and then apply the contraction mapping theorem, following the main idea of \cites{Hong2,Hong3}. \\
\begin{prop}\label{prop1}
Assume $\omega_B$ is cscK with conical singularities along $D$ with angles $\b$ and that $Lie(Aut_D(B,[\omega_B]))$ is trivial. Assume the holomorphic vector bundle $E$ is parabolic stable and equipped with $h_E$ Hermitian-Einstein cone metric obtained via Theorem \ref{HEsol}. Fix $p>0$. There exist deformations of
\begin{itemize}
 \item the form $\omega_B$ as $\omega_B+\iddbar \eta_{k,p}$, with $$\eta_{k,p}=\eta_0 + \eta_1k^{-1}+...+\eta_{p-2}k^{-p+2}\in C^{4,\alpha,\beta}(B\setminus D),$$
 \item  the Hermitian-Einstein metric $h_E$ as $h_E(Id_E+ \Phi_{k,p})$ with $$\Phi_{k,p}=\Phi_0k^{-1} + ... +\Phi_{p-2}k^{-p+1}\in C^{2,\alpha,\b}(Herm(E,h_E)),$$
\end{itemize}
such that the induced K\"ahler form $\tilde{\omega}_k\in [k\pi^*\omega_B+\hat{\omega}_E]$ on $X$ can be deformed to obtain an almost cscK cone metric outside $\mathcal{D}$ of order $p+1$, i.e there exists real valued functions on $X\setminus \mathcal{D}$,
$$\phi_{k,p}=\phi_0k^{-2} + ... +\phi_{p-2}k^{-p}\in C^{4,\alpha,\beta}(X\setminus \mathcal{D}),$$ such that  over $X\setminus \mathcal{D} $, we have
$$S(\tilde{\omega}_k + \sqrt{-1}\partial\bar\partial \phi_{k,p})=\underline S_\beta+O(k^{-p-1})$$
where $\underline S_\beta$ is the topological constant.
\end{prop}

We construct the approximation solution by the implicit function theorem inductively, using \lemref{isomorphismtangent}, \lemref{associate} and Proposition \ref{isomorphismvertical}. The proof is given at the end of this section, page \pageref{proofProp1}, and requires several preliminary results. 
The deformation of the scalar curvature is divided into three parts, the function on $B$, the section $\Gamma(B,W)$ and the function on $\mathbb{P}E^*$. Actually, we need to deform all the metrics $\omega_B,h_E$ and $\omega_k$.

Firstly, we need to understand the deformation of the cscK equation on $X$ with respect to $\omega_B$ and $h_E$, where $\omega_B\in \mathcal H_{\b}(B,D)$ is a K\"ahler cone metric on $B$  and $h_E$ a hermitian metric on $E$ compatible to the parabolic structure with respect to $\omega$.
In order to do so, we are going to study the maps
\begin{align*}A_1(\omega_B,h_E)&=S(\omega_B)Id_E +\frac{\sqrt{-1}}{2\pi}\Lambda_{\omega_B} [F_{h_E}]^0\in C^{,\a,\b}(End(E))\\
 S_1(\omega_B,h_E)&={\tr}\left(A_1(\omega_B,h_E) \frac{v\otimes v^{*_{h_E}}}{\Vert v\Vert^2}\right)\in C^{,\alpha,\b}.
\end{align*}
When $\omega_B$ has constant scalar curvature and $h_E$ is Hermitian-Einstein, then  the linearization of $A_1$ at $(\omega_B,h_E)$ is given by
\begin{equation}\label{DA1}
 DA_1(\eta,\Phi)=({\mathbb{L}\mathrm{ic}}_{\omega_B}\eta ) Id_E + \frac{\sqrt{-1}}{2\pi}[\Lambda_{\omega_B}\bar{\partial}\partial \Phi + 2\Lambda_{\omega_B}^2(F_{h_E}\wedge \sqrt{-1}\partial\bar{\partial}\eta)]^0
\end{equation}
where $\eta\in  C^{4,\alpha,\b}$ and $\Phi$ is a hermitian endomorphism with respect to $h_E$. 

\begin{lem}\label{isomorphismtangent}
 Suppose that $Lie(Aut_{{D}}(B,[\omega_B]))$ is trivial and $E$ is parabolic stable with respect to $\omega_B$. Set $C^{\infty}(End(E))^0$ the space of trace-free endomorphisms of $E$ and ${C^{4,\alpha,\b}_{0}}(B)$ the space of $C^{4,\alpha,\beta}$ real functions on $B$ with vanishing integral with respect to $\omega_B$. Then the map
$$ \begin{array}{rl}
 DA_1\hspace{-0.1cm}:\hspace{-0.1cm}{C^{4,\alpha,\beta}_{0}}(B)\oplus C^{2,\alpha,\beta}(B,End(E))^0 \hspace{-0.25cm} &\rightarrow {C^{,\alpha,\beta}_{0}}(B) \oplus C^{,\alpha,\beta}(B,End(E))^0 \\
     (\eta,\Phi) \hspace{-0.25cm} &\mapsto DA_1(\eta,\Phi)
\end{array}$$
defined by \eqref{DA1}, is an isomorphism.
\end{lem}

\begin{proof}

{ If $DA_{1}(\eta,\Phi)=0$, then we get the system of decoupled equations  on $B\setminus D$ by considering the trace part, 
\begin{eqnarray}
 {\mathbb{L}\mathrm{ic}}_{\omega_B}\eta &=&0,\\
 \frac{\sqrt{-1}}{2\pi}[\Lambda_{\omega_B}\bar{\partial}\partial \Phi + 2\Lambda_{\omega_B}^2(F_{h_E}\wedge \sqrt{-1}\partial\bar{\partial}\eta)]^0 &=& 0\label{solvePhi}.
\end{eqnarray}
Solutions of this system live on $B$. Using Theorem \ref{Fredholm},
the first equation has for solution only the constants by assumption on $Lie(Aut_{{D}}(B,[\omega_B]))$. But then $\eta=0$ since $\eta$ has vanishing integral. From this fact, this leads from the second equation of the system, to 
$$[\Lambda_{\omega_B}\bar{\partial}\partial \Phi]^0=0$$
and since $\Phi$ is trace-free and $C^{2,\alpha,\beta}(B,End(E))^0$, 
\begin{equation}\bar{\partial}^*\bar{\partial}\Phi=\sqrt{-1}\Lambda_{\omega_B}\bar{\partial}\partial \Phi =0\end{equation}
and then we get $\int_B \vert \bar{\partial}\Phi\vert^2\frac{\omega_B^n}{n!}=0$ and so $\Phi\in H^0(B,End(E))$, and thus constant from Lemma \ref{lem1}. Now, this constant vanishes since $\Phi$ is trace-free. 
Eventually we apply Fredholm alternative.
}
From the first equation, we are able to solve $\eta$ via the theory of the Lichnerowicz equation (\thmref{Fredholm}). Putting it into the second equation, we could solve each component of $\Phi$, i.e. \eqref{solvePhi}, by the theory of second order elliptic equations with conical singularities (see Section \ref{Second order elliptic equations with conical singularities}).
\end{proof}

Secondly, we deform the cscK equation on $X$ with respect to $\omega_k$.

\subsection{Decomposition of the holomorphic tangent bundle}

In the sequel of this section, we assume that we are under the setting of Proposition \ref{prop1}, and $h_E$ is a Hermitian-Einstein cone metric.
Given $\hat{\omega}_E$ a form induced by the $ h_E$ on $E$, one can define the following operator $\Delta_V$ on $C^{2,\alpha,\beta}$ functions on the ruled manifold $X\setminus \mathcal{D}$,
$$\Delta_V f \hat{\omega}_E^{r-1}\wedge \pi^* \omega_B^n = (r-1)\iddbar f \wedge \hat{\omega}_E^{r-2}\wedge \pi^* \omega_B^n $$
where $n$ is the dimension of the base $B$ and $\pi:X\to B$ is the projection onto $B$.
\begin{rem}
 It is not a Laplacian. Nevertheless, once restricted to the fibers, it is a Laplacian with respect to the Fubini-Study metric. 
\end{rem}

Now, we know that for a general K\"ahler form $\omega$, the linearization of the scalar curvature is given by
$$\widetilde{\mathbb{L}\mathrm{ic}}_{\omega}(\phi)=(\Delta^2 -S(\omega)\Delta)\phi + n(n-1)\frac{\iddbar \phi\wedge Ric(\omega)\wedge \omega^{n-2}}{\omega^n}.$$
Thus for a smooth function $\phi$ the linearization of the operator the scalar curvature operator at the metric $\omega_k$ is given by
$$ \widetilde{\mathbb{L}\mathrm{ic}}_{\omega_k} = \Delta_V(\Delta_V -r)+ O(k^{-1}),$$
when $k$ tends to $+\infty$ using \eqref{eqS}.
We deform $\omega_k$ and obtain
\begin{align}
S(\tilde{\omega}_k+k^{-2}\sqrt{-1}\partial\bar\partial \phi_{k,1})=&S(\tilde{\omega}_k) +k^{-2} \Delta_V (\Delta_V -r)\phi_{k,1}+O(k^{-3}), \nonumber \\
=&r(r-1)+k^{-1}S_{1}(\omega_B,h_E)+k^{-2}DS_1(\eta_0,\Phi_0) \label{order2a}\\
&+k^{-2}\Delta_V (\Delta_V -r)\phi_{k,1} +O(k^{-3}).  \label{order2b}
\end{align}

Consider the metric $\hat{\omega}_E$ and its Laplacian acting on functions on $\mathbb{P}E^*_x$. We denote $W_x$ as the space of all eigenfunctions associated to the first nonzero eigenvalue of this Laplacian, which is $r=\rk(E)$.

This defines a vector bundle $W$ over $B$.  We have the following result.

\begin{lem}\label{associate}
To any trace-free hermitian endomorphism $$\Phi\in C^{,\a,\b}(End(E))^0\cap Herm(E,h_E),$$ one can associate $$\tr(\Phi\frac{v\otimes v^{*_{h_E}}}{\Vert v\Vert^2})\in C^{,\a,\b}(B)$$ such that its restriction over $B\setminus D$ belongs to $C^{,\a,\b}(B\setminus D,W)$ (i.e is a eigenfunction). The converse is also true, i.e given a eigenfunction that belongs to 
$C^{,\a,\b}(B\setminus D,W)\cap C^{,\a,\b}(B)$, we obtain a trace free hermitian  endomorphism in $ C^{,\a,\b}(B,End(E))^0.$ 
\end{lem}
\begin{proof}
It is known the eigenfunctions of the Laplacian on $\mathbb{CP}^{r-1}$  with respect to the Fubini-Study metric that are associated to the first non zero eigenvalue of the Laplacian. These eigenfunctions are given by harmonic polynomials on $\mathbb{C}^r$ that correspond to certain hermitian endomorphisms, see \cite{MR704609}, \cite[Proposition 2.4]{MR3107685}. We apply this correspondence over $B\setminus D$. Consequently, we obtain that given $\Phi$ as above, $\tr(\Phi\frac{v\otimes v^{*_{h_E}}}{\Vert v\Vert^2})\in C^{,\a,\b}(B)$ is an eigenfunction over $B\setminus D$ and the regularity is clear. 
Now, for the converse, starting with an eigenfunction, we get an endomorphism $\Phi_1$ over $B\setminus D$ as before. We extend the subbundle $Im(\Phi_1)$ as a subbundle of $C^{,\a,\b}(B,End(E))$. For doing that, we apply \cite[Lemma 10.6]{Sim}. We just need to see that the curvature of $End(E)$ is $L^1$, but this is the case as $F_{End(E)}=F_{E,h_E}\otimes Id_{E^*} + Id_E\otimes F_{E^*,h_E^*}$ and $h_E$ is Hermitian-Einstein.
\end{proof}

\begin{rem}
 With the lemma in hands, we see why it is not sufficient to deform the metrics $\omega_B,h_E$ and that we need to also deform the metric $\omega_k$.
\end{rem}

Now, for $p=2$ or $p=4$, one can consider the orthogonal space to $C^{p,\alpha,\beta}(B\setminus D,W)$ of functions in $C^{p,\alpha,\beta}(X\setminus \mathcal{D})$ with vanishing integral, i.e
\begin{align*}\mathcal{Z}^{p,\a,\b}_0=\{ &\phi \in C^{p,\alpha,\beta}(X\setminus \mathcal{D}), \int_{\mathbb{P}E^*_x} \phi\hat{\omega}_E^{r-1}=0  \\
&\text{ and }  \int_{\mathbb{P}E^*_x} \phi \theta\hat{\omega}_E^{r-1}=0, \forall x\in B\setminus D,\theta \in C^{p,\alpha,\beta}(B\setminus D,W)\}.\end{align*} 

We consider the map
$$\begin{array}{rl}
 L_V=\Delta_V(\Delta_V-r) : \mathcal{Z}^{4,\a,\b}_0 &\longrightarrow \mathcal{Z}^{,\a,\b}_0 \\\phi &\longmapsto \Delta_V(\Delta_V-r)(\phi).
\end{array}$$

\begin{lem}\label{LVinjective}
$L_V$ is an injective.
\end{lem}
\begin{proof}
For the injectivity, assume that $\Delta_V(\Delta_V-r)\phi =0$. Then over each fiber above $x\in B\setminus D$, $\Delta_V(\Delta_V-r) \phi_x=0$. By compactness of the fiber, it implies that 
$(\Delta_V-r)\phi$ is constant and since $\phi$ has vanishing integral, we get $(\Delta_V-r)\phi=0$ over $B\setminus D$. But then $\phi$ is an eigenfunction associated to the $r$-th eigenvalue, and thus by orthogonality, $\phi=0$.

\end{proof}

Next we prove that the map $L_V$ is also surjective. We adapt an idea from the theory of differential families of strongly elliptic differential operators, see Section 7 in Kodaira \cite{MR2109686} for instance.

We first cut out a small open neighbourhood $D_\delta$ containing $D$ with radius $\delta$. On each fibre at $x\in B\setminus D_\delta$, the ellipticity of the operator $\Delta_V(\Delta_V-r)(\phi)$ tells us that 
there is a complete orthonormal set of eigenfunctions $\{e_i(x)\}$. Each corresponding eigenvalues $\lambda_{i}(x)$ is a continuous function of $x\in B\setminus D_\delta$. 

Given $\psi\in \mathcal{Z}^{,\a,\b}_0$ we construct a solution $\phi_x$ for all $x\in B\setminus D$ of
$$L_{V,x}\phi_x=\Delta_{V,x}(\Delta_{V,x}- r)\phi_x=\psi_x$$
over $\mathbb{P}E^*_x$.
Along each fibre, we are able to solve the equation $L_{V,x}\phi_x=\psi_x$, since the operator is strongly elliptic and self-adjoint with respect to the inner product $ \int_{\mathbb{P}E^*_x}(\cdot,\cdot)\hat{\omega}_E^{r-1}$. 

We define the spaces $C_V^{k,\a}$, $k\geq 0$ as the H\"older spaces defined on the fibre $\mathbb{P}E^*_x$. Along each fibre, we then prove a priori estimates, with the help of the injectivity of the operator $L_V$.
\begin{lem}\label{familyestimate}
Suppose that $\phi\in \mathcal{Z}^{4,\a,\b}_0$ in $\mathbb{P}E^*$. There exists a constant $C>0$ depending on $\mathbb{P}E^*_x$ and the coefficients of $L_{V,x}$ such that
\begin{align*}
|\phi_x|_{C^{4,\a}_V}\leq C |L_{V,x}\phi_{x}|_{C_V^{\a}}.
\end{align*}
\end{lem}
\begin{proof}
Since $L_{V,x}$ is strongly elliptic, we have
\begin{align*}
|\phi_x|_{C^{4,\a}_V}\leq C |L_{V,x}\phi_{x}|_{C_V^{\a}}+|\phi_x|_{C^{,\a}_V}.
\end{align*}
We now assume the conclusion is not true. Then there exists a sequence of $\phi_x(k)\in \mathcal{Z}^{4,\a,\b}_0$, $k\geq1$ such that 
\begin{align*}|\phi_x(k)|_{C^{,\a}_V}=1, \forall k\geq 1\text{ and } |L_{V,x}\phi_{x}(k)|_{C_V^{\a}}\rightarrow 0, \text{ as } k\rightarrow \infty.
\end{align*} 
Thus from compactness of $C^{4,\a}_V$, 
$|\phi_x(\infty)|_{C^{,\a}_V}=1$ and $L_{V,x}\phi_{x}(\infty)=0$. But $\mathcal{Z}^{4,\a,\b}_0$ is closed and then we use \lemref{LVinjective}, there is no nontrivial kernel in $\mathcal{Z}^{4,\a,\b}_0$. Thus $\phi_x(\infty)=0$, which is a contradiction.
\end{proof}

Let $x,y$ be two points in $B\setminus D_\delta$.
We let $\phi_x,\phi_y$ solve the equations $L_{V,x}\phi_x=\psi_x$ and $L_{V,y}\phi_y=\psi_y$ respectively. Now we are ready to prove in the next two lemma, the asymptotic behaviors of the solution $\phi_x$ with respect to the point $x$ in the base manifold $B\setminus D$.

\begin{lem}\label{familycontinuity}
The family of solution $\phi_x$ is continuous in $x$.
\end{lem}
\begin{proof}
We have 
\begin{align*}
L_{V,x}(\phi_y-\phi_x)+(L_{V,y}-L_{V,x})\phi_{y}=\psi_y-\psi_x.
\end{align*}
From \lemref{familyestimate}, we have
\begin{align*}
|\phi_y-\phi_x|\leq C |(L_{V,y}-L_{V,x})\phi_{y}|_{C_V^{\a}}+|\psi_y-\psi_x|_{C_V^{\a}}.
\end{align*}
Here $C$ depends on the geometry of the fibre $\mathbb{P}E^*_x$ and the coefficients of $L_V$ at $x$. Since the coefficients of $L_V$ and $\psi$ are both $C^{,\a,\b}$, we see that $\phi_y\rightarrow\phi_x$ as $y\rightarrow x$.
\end{proof}

\begin{lem}\label{familyc1}
The family of solution $\phi_x$ is $C^{1,\b}$ in $x$.
\end{lem}
\begin{proof}
We denote $\nabla_x$ be the covariant derivative with respect to the metric $\pi^\ast\omega_B$.
We let $\eta_x$ solves the equation 
\begin{align*}
L_{V,x}\eta_x=\nabla_x \psi_x -(\nabla_x L_{V,x})\phi_x.
\end{align*}
We also denote distance $\delta$ between $x,y$ are measured under the metric $\pi^\ast\om_B$. It is sufficient to prove that as $y\rightarrow x$,
\begin{align*}
\frac{\phi_y-\phi_x}{\delta}\rightarrow \eta_x.
\end{align*}
From \lemref{familycontinuity}, it suffices to prove that
as $y\rightarrow x$,
\begin{align*}
L_{V,y}[\frac{\phi_y-\phi_x}{\delta}- \eta_x]\rightarrow 0.
\end{align*}
Using the formulas above, it becomes
\begin{align*}
L_{V,y}[\frac{\phi_y-\phi_x}{\delta}- \eta_x]
=&\frac{\psi_y-\psi_x}{\delta}-\nabla_x\psi_x\\
&-[\frac{(L_{V,y}-L_{V,x})\phi_x}{\delta}-(\nabla_xL_{V,x})\phi_x]\\
&-(L_{V,y}-L_{V,x})\eta_x.
\end{align*}
The first two terms on the right hand side converge to zero by definition and the third term converges to zero, since the coefficients are in $C^{,\a,\b}$.

\end{proof}

\begin{prop}\label{isomorphismvertical}
$L_V$ is bijective.
\end{prop}
\begin{proof}
Thanks to Lemma \ref{LVinjective}, we just need to prove that $L_V$ is surjective. 
The problem now is how to glue the family of pointwise solutions together to produce a solution on the whole ruled manifold. Combining \lemref{familycontinuity} and \lemref{familyc1}, we obtain $\phi\in C^{4,\a,\b}$ by using the induction argument and repeating the strategy of \lemref{familyc1}.
\end{proof}

\begin{rem}
 Note that the proofs of Lemmas \ref{LVinjective}, \ref{familyestimate}, \ref{familycontinuity}, \ref{familyc1} and Proposition \ref{isomorphismvertical} do not require $h_E$ to be Hermitian-Einstein but only to have the regularity obtained in Theorem \ref{HEsol}.
\end{rem}

\label{proofProp1}\begin{proof}[Proof of Proposition \ref{prop1}] The first two terms of the expansion of the scalar curvature $S(\tilde{\omega}_k+ \sqrt{-1}\partial\bar\partial\phi_{k,p})$ are constant since
 $S_1(\omega_B,h_E)$ is constant by definition of $\omega_B$ and $h_E$. Now, 
we wish to make constant the $k^{-2}$ term of the expansion of $S(\tilde{\omega}_k+ \sqrt{-1}\partial\bar\partial\phi_{k,p})$ constant. For that, writing the topological constant  $\underline{S}_\beta=\underline{S}_\beta^0+k^{-1}\underline{S}_\beta^1+ k^{-2}\underline{S}_\beta^2+...$ and considering  \eqref{order2a}, \eqref{order2b},  we need to find $(\eta_0,\Phi_0,\phi_0)$ solving the equation over $X\setminus \mathcal{D}$,
$$DS_1(\eta_0,\Phi_0)+\Delta_V (\Delta_V -r)\phi_{0} = \underline{S}_\beta^2$$
Since $\underline{S}_\beta^2$ is a constant, we are lead to solve both equations $\Delta_V (\Delta_V -r)\phi_{0}=0$ and $DS_1(\eta_0,\Phi_0)=\underline{S}_\beta^2$.
Note that first equation has an obvious solution.
Hence, we can apply Lemma \ref{isomorphismtangent} and Lemma \ref{associate} to obtain $(\eta_0,\Phi_0)$. \\
Next, we need to make constant the $k^{-3}$ term. As there is a contribution coming from expansion of
$S(\tilde{\omega}_k)$, we are lead this time to solve  
$$DS_1(\eta_1,\Phi_1)+\Delta_V (\Delta_V -r)\phi_{1} =  \underline{S}_\beta^3+\gamma_3$$
for a certain function $\gamma_3\in C^{,\alpha,\beta}(X\setminus \mathcal{D},\mathbb{R})$. Here we use the fact that we found $\eta_0$ and $ \phi_0$ in $C^{4,\a,\b}$ and $\Phi_0 \in C^{2,\a,\b}$. The term 
$\gamma_3$ appears as a combination of 4th order derivatives of $\eta_0,\phi_0$ and 2{nd} order derivatives in $\Phi_0$. Consequently, from the fact that $C^{,\a,\b}$ is an algebra, $\gamma_3$ is an element of $C^{,\alpha,\beta}(X,\mathbb{R})$.
Again, Lemma \ref{isomorphismtangent}, Lemma  \ref{associate} and Proposition \ref{isomorphismvertical} ensure that a solution $(\eta_1,\Phi_1,\phi_1)$ can be found with the right regularity.

Then Proposition \ref{prop1} is obtained by induction using the same reasoning for higher order terms, thanks to the fact that we have surjectivity from 
 $C^{4,\alpha,\beta}_0(B)\times C^{2,\alpha,\beta}(B,End(E))^0\times \mathcal{Z}^{p,\a,\b}_0$ onto $C^{,\alpha,\beta}(X\setminus \mathcal{D},\mathbb{R})$ and applying  Lemmas \ref{isomorphismtangent} and  \ref{associate} and Proposition \ref{isomorphismvertical}.

\end{proof}

\subsection{Proof of Theorem \ref{thm1}}\label{proofThm1}

\begin{proof}[Proof of Theorem \ref{thm1}]
Note that the proof is technically different to \cite{LS}, as in the smooth case it is used Sobolev spaces while we use H\"older spaces. Since $E$ is parabolic stable, we obtain a Hermitian-Einstein cone metric from Theorem \ref{HEsol}. We can apply Proposition \ref{prop1} to obtain an approximate cscK cone metric. To obtain a genuine cscK cone metric, we need to apply the implicit function theorem using Theorem \ref{Fredholm} on $X$ the projectivisation of the bundle. This requires to solve the Lichnerowicz equation as the kernel of the associated operator is trivial, thanks to Corollary \ref{cornovectorfield}.
\end{proof}


\section{K\"ahler-Einstein cone metrics and Tangent bundle}\label{KET}

It is well-known that if $X$ is a compact K\"ahler manifold endowed with a smooth K\"ahler-Einstein metric (with positive, negative or zero curvature) then its tangent bundle $TX$ admits a Hermitian-Einstein metric and thus is Mumford polystable (with respect to the anticanonical polarization, canonical polarization, or any polarization respectively). In this section, we study the case when $X$ is a compact K\"ahler manifold and admits a K\"ahler-Einstein metric with conical singularities along a divisor. Let $n$ be the complex dimension of $X$.\\
Let us consider $\omega_{KE}$ a K\"ahler-Einstein metric with conical singularities along $D$ smooth divisor for which the H\"older exponent $\alpha$ and the angle $2\pi\beta$ satisfy Condition \eqref{anglerestriction}. Define $\nabla_{KE}$ the  Chern connection associated to the induced hermitian metric $h_{KE}$ on the tangent bundle $TX$ and  $F_{KE}=F_{\nabla_{KE}}\in \Omega^{1,1}(End(T^{1,0}X))$ its curvature. We obtain an operator 
$$\widehat{\omega}_{KE}(\Lambda_{\omega_{KE}}F_{{KE}}):T^{1,0}X\to \Lambda^{0,1}X$$
by identifying $T^{1,0}X$ to $\Lambda^{0,1}X$ using $\omega_{KE}$. Consequently, $\widehat{\omega}_{KE}(\Lambda_{\omega_{KE}}F_{{KE}})$ can be seen as an element in $\Omega^{1,1}(X)$. A local computation that remains valid outside of the divisor $D$, shows that $$Ric(\omega_{KE})=\widehat{\omega}_{KE}(\Lambda_{\omega_{KE}}F_{{KE}}).$$ 
Using the K\"ahler-Einstein property, the previous equation provides a metric on $TX$ which is Hermitian-Einstein metric outside $D$ and has 
$C^{2,\a,\b}$ regularity as $Ric(\omega_{KE})$ is $C^{,\a,\b}$ from Theorem \ref{cscKconeregularity}. We now check that we obtain furthermore a parabolic structure for which this metric is compatible.
We consider the canonical section $\sigma$ of $D$ that vanishes precisely on $D$.
Now, as the K\"ahler form $\omega_{KE}$ has conical singularities along $D$, it is quasi-isometric to 
\begin{equation}\frac{\sqrt{-1}}{2} a_{1}\vert \sigma\vert^{2(\beta-1)}dz^1\wedge d\bar{z}^1+\tilde{\omega}\label{curvKE}\end{equation}
using local cone chart coordinates. Here the $z^1$ is the local defining function of the hypersurface $D=\{\sigma=0\}$ where $p$ locates, $a_1$ is a smooth function, $0<\beta<1/2$, and $\tilde{\omega}$ is a smooth form.
Consequently, from \eqref{curvKE}, the curvature $\vert F_{KE}\vert_{h_{KE}}$ of the metric $h_{KE}$ lies in $L^p(X)$ for $p< \frac{2}{1-\beta}$ as 
\begin{align}\label{ineq000}\vert \sigma\vert^{1-\beta}\vert F_{KE}\vert_{h_{KE}}\leq C,\end{align}
for a uniform constant $C>0$, see for instance the proof of \cite[Lemma 5.2]{MR1701135}.
This bound implies the following statement. 
\begin{lem}\label{Lpcurv}
 With above assumptions, there exists $C>0$ such that
 $$\Vert F_{{KE}} \Vert_{L^p(\omega_{KE})}<C,$$ 
with $p>2$. 
\end{lem}
\begin{thm}\label{KEHE}
 Assume $(X,\omega_{KE})$ is a compact K\"ahler manifold endowed with a K\"ahler-Einstein cone metric along a smooth divisor $D$, with H\"older exponent $\alpha$ and angle $2\pi\beta$ satisfying Condition \eqref{anglerestriction}. Then its tangent bundle $TX$ is parabolic polystable with respect to $\omega_{KE}$.
\end{thm}
\begin{proof}
 Firstly, note that in the case of a curve or a surface, we can use a strong result of Biquard that shows an equivalence between the category of hermitian bundles on $X\setminus D$ with $L^p$ curvature and the category of holomorphic bundles on $X$ with parabolic structure over $D$, see \cites{MR1168354,MR1373062}. Applying Lemma \ref{Lpcurv}, we obtain that the  tangent bundle can be extended over $D$ together with a parabolic structure along the divisor and such that the metric $h_{KE}$ is compatible with this parabolic structure. Moreover this extension is essentially unique. Now, using the Hermitian-Einstein condition and \cite[Theorem 6.3]{LiJ} or \cite{Sim}, we obtain the parabolic stability of each component of the  tangent bundle if it is not indecomposable, i.e its polystability.

 In general, we can adapt the construction of \cite[Section 3]{LiJ} as the bundle we are interested in is already defined over the divisor. We choose local holomorphic coordinates in a neighbourhood $U=\{\vert z^i\vert<1, i=1,..,n\}$ of the point $p=(0,...,0)$ such that the intersection with the divisor can be written $D\cap U=\{z^1= 0\}$.
 We may choose a holomorphic basis $\{e_i\}_{i=1,..,r}$ of $E_{\vert U}$ such that the matrix of the metric $h_{KE}$ in this basis is diagonal. We write $D_{h_{KE}}$ this matrix.
 The matrix   $D_{h_{KE}}$ necessarily vanishes on $U\cap D$ as the curvature $F_{KE}$ is singular and thus, fixing $\Vert . \Vert_D$ a norm on $\mathcal{O}(D)$, we can write $D_{h_{KE}}$ as
 $$D_{h_{KE}}=\begin{pmatrix}
               \zeta_1\Vert z^1\Vert^{2\gamma_{1}}_{D}&  & \\ 
                & \ddots & \\
                & & \zeta_r\Vert z^1\Vert^{2\gamma_{r}}_{D} 
              \end{pmatrix}.$$
Here $\zeta_j$ are positive smooth functions, the $\gamma_j$ are non-negative real numbers and \eqref{ineq000} ensures that $\gamma_j<1$ for all $j$. Without loss of generality we can assume that the $\gamma_j$ is an increasing sequence, by doing 
a permutation of $\{e_i\}$. We denote $r_j$ the integers which count the numbers of equal $\gamma_j$, i.e that $r_1=\rk(TM)=n$ and $r_{j+1}$ is defined inductively by $\gamma_l=\gamma_{n-r_{j}+1}$ for $n-r_{j}+1\leq l\leq n-r_{j+1}$. Let $l_E$ be the number of different integers $r_i$. We define $\mathcal{F}_{\vert D\cap U}^i=Vect(e_{n-r_i+1},...,e_{n})$ for $i=2,..,l_E$ and $\alpha_i=\gamma_{n-r_{i+1}}$. Clearly, the data  $(\mathcal{F}^i,\alpha_i)$ defines a parabolic structure for $E_{\vert D \cap U}$.  If we consider another neighbourhood $U'$ that intersects $U \cap D$, then one can find as above a basis $\{e'_i\}_{i=1,..,r}$ of $E_{\vert U'}$ such that the matrix of $h_{KE}$ is diagonal with diagonal entries $\zeta'_i\Vert \sigma\Vert^{2\gamma'_{i}}_{D}$ for $i=1,..,r$ with $\gamma'_i$ increasing sequence and $\sigma$ the canonical section of $D$ which vanishes precisely on $D$. The vanishing order must be the same on $U\cap U'\cap D\neq \emptyset$ which forces $\gamma_i=\gamma_i'$ and consequently $\mathcal{F}_{\vert D\cap U\cap U'}$ extends to $\mathcal{F}_{\vert D\cap U'}$. Consequently we have defined a filtration by subbundles of $E_{\vert D}$ and thus a  parabolic structure for $E$ along $D$. Then, following \cite[Section 3]{LiJ}, the metric $h_{KE}$ is compatible with the structure and $TM$ is endowed with a Hermitian-Einstein cone metric. We conclude as above using Remark \eqref{remHE}.
 \end{proof}


\section{Further Applications and Remarks}\label{remarksection}

\subsection{Simple normal crossings divisors}

We expect that Theorems \ref{Fredholm}, \ref{thm1}, \ref{closedness}, \ref{HEsol}, \ref{KEHE} and Corollary \ref{cor1} could be generalized to the case of simple normal crossings divisors $D=\sum_{i=1}^m D_i$ where $D_i$ are irreducible. It is possible to define H\"older and Sobolev spaces for K\"ahler cone metrics $\omega$ with angles $2\pi\beta_i$ along $D_i$. The condition \eqref{anglerestriction} would be replaced by the condition \eqref{anglerestriction2}
\begin{align}\label{anglerestriction2}\tag{C'}
0<\beta_i<\frac{1}{2}; \quad \quad \a\b_i<1-2\b_i,\, \forall i=1,...,m.
\end{align}
The statements of our results would remain identical under above changes. The only missing step is a Schauder estimate that extends Proposition \ref{Schauder}. This has been announced recently in the preliminary work of Guo-Song \cite{GS}.

\subsection{Twisted conical path for cscK metric}
Donaldson introduced a continuity method for conical K\"ahler-Einstein metrics \cite[Equation (27)]{MR2975584} on Fano manifolds. A natural extension of this path for general polarizations is given by the following equation, that we call {\it scalar curvature twisted conical path}.
\begin{align}\label{Do-eqnvar}
S(\om_{\vphi(t)})=c_t+2\pi(1-t) \tr_{\om_{\vphi(t)}}[D].
\end{align}
where $c_t$ is a constant that depends on the time given in terms of topological invariants, $[\om_{\vphi(t)}]=2\pi c_1(L)$ for $L$ ample line bundle on the projective manifold $X$. Note that this is a variant of the continuity method introduced by Chen \cite[Equation (2.16)]{chen}. Scalar curvature twisted conical path is expected to be helpful for the construction of smooth cscK metric as $t\to 1$. It is natural to ask if the set of times $t$ for which \eqref{Do-eqnvar} admits a solution is open and non empty. As far as we know, the existence of a solution at an initial time $t_0>0$ is not known. Nevertheless, \cite[Corollary 5.10]{MR3403730} shows the existence of an invariant $\beta(X,D)$ such  for $0<t_0<\beta(X,D)$, $(X,D)$ is log K-stable for angle $t_0$. 
If an initial solution does exist at time $t_0<\min(1/2, \beta(X,D))$, Theorem \ref{Fredholm} applies and provides the openness property if $Lie(Aut_D(X,L))$ is trivial.
This is similar to the smooth case, see \cite[Theorems 1.5 and 1.8]{chen}.

\subsection{Other applications of Theorem \ref{Fredholm}}
The linear theory proved in this paper could be used to construct a large family of cscK cone metrics over blow-ups, extending the previous work of Arezzo-Pacard 
\cites{MR2275832}. We also expect a generalization of the work of Fine (see \cite{Fine} and subsequent works) for construction at the adiabatic limit of cscK cone metrics over a holomorphic submersion between compact K\"ahler manifolds $\pi:X\to B$ where the fibers and the base do not admit non trivial holomorphic vector fields and each fiber admits a cscK cone metric. 

\subsection{Generalizations of Theorem \ref{thm1}}
In view of \cites{Hong3,LS,Bro,MR2807093,LiH} it is natural to ask whether Theorem  \ref{thm1} admits a generalization in the case the base manifold $B$ admits nontrivial holomorphic vector fields. As in the smooth case, this will not be automatically happen,  and an extra condition on $(B,D)$ will be required.  A related question is about the existence of extremal K\"ahler cone metric when the bundle $E$ splits as sum of parabolic stable bundles of different slopes. These problems will be investigated in a forthcoming paper.

\subsection{Log K-stability}
We expect that a purely algebraic version of Theorem \ref{thm1} holds, i.e under the same assumptions, one gets also log-K-stability of $(X,\mathcal{D},[\omega_k]=[k\pi^*\omega_B+\hat{\omega}_E])$ for large $k$. We refer to \cite{MR2975584} and \cite{MR3313212} for the notion of log-K-stability and the logarithmic version of Yau-Tian-Donaldson conjecture. In the particular case of vector bundles over a curve, by analogy to the smooth case, we expect the following to be true.

\begin{conj}
Let $C$ be a complex curve endowed with a cscK cone metric $\omega_{cscK}$ along a divisor $D$.
 Let $E$ be a parabolic vector bundle over $C$, with parabolic structure along $D$. Let  $X=\mathbb{P}(E^*) \to C$ be its projectivisation and $\mathcal{D}=\pi^{-1}(D)$. The following three conditions  are equivalent:
\begin{enumerate}
\item[\rm (i)] $X$ admits a cscK cone metric along $\mathcal{D}$ in any ample class $2\pi c_1(\cL)$ on $X$;
\item[\rm (ii)] $(X,\mathcal{D})$ is log-K-polystable for any polarization $\cL$  on $X$;
\item[\rm (iii)] $E$ is parabolic polystable with respect to $\omega_{cscK}$, i.e it decomposes as the sum of stable parabolic bundles of same parabolic slopes.
\end{enumerate}
\end{conj}

Before we provide some information on this conjecture, let us mention that from the work of Troyanov we have an algebraic characterization of complex curves that can be endowed with a cscK cone metric, see \cite{MR1005085} and \cite{MR3313212} for the relation with log K-stability. \\

{\bf \textrm{(iii)}$\Rightarrow$\textrm{(i)}}. At the boundary of the K\"ahler cone and when $E$ is irreducible, on can invoke Theorem \ref{thm1} which provides a more precise result than \cite{Ke1} in terms of regularity of the cscK cone metric. Nevertheless the implication \textrm{(iii)}$\Rightarrow$\textrm{(i)} is true in general as soon as condition \eqref{anglerestriction} holds. Actually, starting from  $\omega_B$ cscK cone metric and $h_E$ Hermitian-Einstein metric on $E$, the K\"ahler cone metric $\omega_k$ given by \eqref{omegak} has constant scalar curvature outside $\mathcal{D}$, like in the smooth case (to check this fact one can also refine Lemma \ref{Hongt} by obtaining a complete expansion when the base has dimension 1). Moreover $\omega_k$ has $C^{2,\a,\b}$ potential from Lemma \ref{regomegak} and thus is a genuine cscK cone metric. Note that for this reasoning we don't need to assume $k>>0$ and that when the base is a curve, any K\"ahler class on $X$ is of the form $[\omega_k]$.\\

{\bf \textrm{(i)}$\Leftrightarrow$\textrm{(ii)}}.  Hashimoto \cite{Hasfutaki} provides essentially the equivalence for a very particular bundle $E$ and $C=\mathbb{P}^1$. Remark that the considered bundle $E$ can be made parabolic polystable by using the techniques \cite{Ke1}. Moreover, Corollary \ref{logFutvanish} gives evidence that the implication \textrm{(i)}$\Rightarrow$\textrm{(ii)} holds as product test configurations correspond to holomorphic vector fields with a holomorphy potential.\\

{\bf \textrm{(ii)}$\Rightarrow$\textrm{(iii)}}. A weaker version is shown in terms of asymptotic Chow polystability with angle in \cite{Ke1}.

\bigskip

\bigskip

{\small
\textbf{Acknowledgments}. The work of J. Keller has been carried out in the framework of the Labex Archim\`ede (ANR-11-LABX-0033) and of the A*MIDEX project (ANR-11-IDEX-0001-02), funded by the ``Investissements d'Avenir" French Government programme managed by the French National Research Agency (ANR). J. Keller is also partially supported by supported by the ANR project EMARKS, decision No ANR-14-CE25-0010. J. Keller thanks O. Biquard and P. Eyssidieux. 

The work of K. Zheng has received funding from the European Union's Horizon 2020 research and innovation program under the Marie Sk{\l}odowska-Curie grant agreement No 703949, and was also partially supported by the Engineering and Physical Sciences Research Council (EPSRC) on a Program Grant entitled {\it Singularities of Geometric Partial Differential Equations} reference number EP/K00865X/1.
}


\begin{bibdiv}
\begin{biblist}

\bib{MR2807093}{article}{
   author={Apostolov, Vestislav},
   author={Calderbank, David M. J.},
   author={Gauduchon, Paul},
   author={T\o nnesen-Friedman, Christina W.},
   title={Extremal K\"ahler metrics on projective bundles over a curve},
   journal={Adv. Math.},
   volume={227},
   date={2011},
   number={6},
   pages={2385--2424},
}


\bib{MR2275832}{article}{
   author={Arezzo, Claudio},
   author={Pacard, Frank},
   title={Blowing up and desingularizing constant scalar curvature K\"ahler
   manifolds},
   journal={Acta Math.},
   volume={196},
   date={2006},
   number={2},
   pages={179--228},
}


\bib{MR1215276}{article}{
   author={Bando, Shigetoshi},
   title={Einstein-Hermitian metrics on noncompact K\"ahler manifolds},
   conference={
      title={Einstein metrics and Yang-Mills connections},
      address={Sanda},
      date={1990},
   },
   book={
      series={Lecture Notes in Pure and Appl. Math.},
      volume={145},
      publisher={Dekker, New York},
   },
   date={1993},
   pages={27--33},
}

\bib{MR3107540}{article}{
   author={Berman, Robert J.},
   title={A thermodynamical formalism for Monge-Amp\`ere equations,
   Moser-Trudinger inequalities and K\"ahler-Einstein metrics},
   journal={Adv. Math.},
   volume={248},
   date={2013},
   pages={1254--1297},
}

\bib{MR1168354}{article}{
    AUTHOR = {Biquard, Olivier},
     TITLE = {Prolongement d'un fibre holomorphe hermitien \`a\ courbure
              {$L^p$} sur une courbe ouverte},
   JOURNAL = {Internat. J. Math.},
    VOLUME = {3},
      YEAR = {1992},
    NUMBER = {4},
     PAGES = {441--453},
      ISSN = {0129-167X},
}

\bib{MR1373062}{article}{
    AUTHOR = {Biquard, Olivier},
     TITLE = {Sur les fibr\'es paraboliques sur une surface complexe},
   JOURNAL = {J. London Math. Soc. (2)},
    VOLUME = {53},
      YEAR = {1996},
    NUMBER = {2},
     PAGES = {302--316},
      ISSN = {0024-6107},
}




\bib{MR3144178}{article}{
   author={Brendle, Simon},
   title={Ricci flat K\"ahler metrics with edge singularities},
   journal={Int. Math. Res. Not. IMRN},
   date={2013},
   number={24},
   pages={5727--5766},
}

\bib{Bro}{article}{
    AUTHOR = {Br\"onnle, T},
     TITLE = {Extremal {K}\"ahler metrics on projectivized vector bundles},
   JOURNAL = {Duke Math. J.},
    VOLUME = {164},
      YEAR = {2015},
    NUMBER = {2},
     PAGES = {195--233},
      ISSN = {0012-7094},
}
  
\bib{MR2746347}{article}{
	author={Boucksom, S\'ebastien},
	author={Eyssidieux, Philippe},
	author={Guedj, Vincent},
	author={Zeriahi, Ahmed},
	title={Monge-Amp\`ere equations in big cohomology classes},
	journal={Acta Math.},
	volume={205},
	date={2010},
	number={2},
	pages={199--262},
}


\bib{MR3405866}{article}{
   author={Calamai, Simone},
   author={Zheng, Kai},
   title={Geodesics in the space of K\"ahler cone metrics, I},
   journal={Amer. J. Math.},
   volume={137},
   date={2015},
   number={5},
   pages={1149-1208},
}

\bib{MR2484031}{article}{
    AUTHOR = {Chel'tsov, I. A.},
    author = {Shramov, K. A.},
     TITLE = {Log-canonical thresholds for nonsingular {F}ano threefolds},
   JOURNAL = {Uspekhi Mat. Nauk},
    VOLUME = {63},
      YEAR = {2008},
    NUMBER = {5(383)},
     PAGES = {73--180},
      ISSN = {0042-1316},
}
%

\bib{chen}{article}{
    AUTHOR = {Chen, Xiuxiong},
    TITLE = {On the existence of constant scalar curvature K\"ahler metric: a new perspective},
      JOURNAL = {},
       VOLUME = {},
     PAGES= {ArXiv:1506.06423.},
       YEAR={},

}


\bib{chenwang}{article}{
    AUTHOR = {Chen, Xiuxiong},
    AUTHOR = {Wang, Yuanqi},
    TITLE = {On the regularity problem of complex Monge-Amp\`ere equations with conical singularities},
      JOURNAL = {},
       VOLUME = {},
     PAGES= {ArXiv:1405.1021.},
       YEAR={},

}

\bib{MR765366}{article}{
   author={Donaldson, S. K.},
   title={Anti self-dual Yang-Mills connections over complex algebraic
   surfaces and stable vector bundles},
   journal={Proc. London Math. Soc. (3)},
   volume={50},
   date={1985},
   number={1},
   pages={1--26},
}


\bib{MR885784}{article}{
   author={Donaldson, S. K.},
   title={Infinite determinants, stable bundles and curvature},
   journal={Duke Math. J.},
   volume={54},
   date={1987},
   number={1},
   pages={231--247},
}
\bib{MR1165874}{article}{
   author={Donaldson, S. K.},
   title={Boundary value problems for Yang-Mills fields},
   journal={J. Geom. Phys.},
   volume={8},
   date={1992},
   number={1-4},
   pages={89--122},
}


\bib{MR2039989}{article}{
    AUTHOR = {Donaldson, S. K.},
     TITLE = {Moment maps in differential geometry},
 BOOKTITLE = {Surveys in differential geometry, {V}ol.\ {VIII} ({B}oston,
              {MA}, 2002)},
    SERIES = {Surv. Differ. Geom.},
    VOLUME = {8},
     PAGES = {171--189},
 PUBLISHER = {Int. Press, Somerville, MA},
      YEAR = {2003},
}

\bib{MR2975584}{article}{
    AUTHOR = {Donaldson, S. K.},
     TITLE = {K\"ahler metrics with cone singularities along a divisor},
 BOOKTITLE = {Essays in mathematics and its applications},
     PAGES = {49--79},
 PUBLISHER = {Springer, Heidelberg},
      YEAR = {2012},
}


\bib{MR3522175}{article}{
    AUTHOR = {Eyssidieux, Philippe},
     TITLE = {M\'etriques de {K}\"ahler-{E}instein sur les vari\'et\'es de {F}ano
              [d'apr\`es {C}hen-{D}onaldson-{S}un et {T}ian]},
   JOURNAL = {Ast\'erisque, S\'eminaire Bourbaki. Vol. 2014/2015},
    NUMBER = {380},
      YEAR = {2016},
     PAGES = {Exp. No. 1095, 207--229},
      ISSN = {0303-1179},
      ISBN = {978-2-85629-836-7},
}

\bib{Fine}{article}{
    AUTHOR = {Fine, Joel},
     TITLE = {Constant scalar curvature {K}\"ahler metrics on fibred complex
              surfaces},
   JOURNAL = {J. Differential Geom.},
    VOLUME = {68},
      YEAR = {2004},
    NUMBER = {3},
     PAGES = {397--432},
}


\bib{MR1814364}{book}{
	author={Gilbarg, David},
	author={Trudinger, Neil S.},
	title={Elliptic partial differential equations of second order},
	series={Classics in Mathematics},
	note={Reprint of the 1998 edition},
	publisher={Springer-Verlag, Berlin},
	date={2001},
	pages={xiv+517},
	isbn={3-540-41160-7},%
}

\bib{MR704609}{article}{
    AUTHOR = {Grinberg, Eric L.},
     TITLE = {Spherical harmonics and integral geometry on projective
              spaces},
   JOURNAL = {Trans. Amer. Math. Soc.},
    VOLUME = {279},
      YEAR = {1983},
    NUMBER = {1},
     PAGES = {187--203},
      ISSN = {0002-9947},
}

\bib{GS}{article}{
    AUTHOR = {Guo, Bin},
    AUTHOR = {Song, Jian},
    TITLE = {Schauder estimates for equations with cone metrics, I},
      JOURNAL = {},
       VOLUME = {},
     PAGES= {ArXiv:1612.00075},
       YEAR={2016},

}

\bib{Hasfutaki}{article}{
    AUTHOR = {Hashimoto, Yoshinori},
    TITLE = {Scalar curvature and Futaki invariant of {K}\"ahler metrics with cone singularities along a divisor},
      JOURNAL = {},
       VOLUME = {},
     PAGES= {ArXiv:1508.02640},
       YEAR={2015},

}

\bib{Has}{article}{
    AUTHOR = {Hashimoto, Yoshinori},
    TITLE = {Existence of twisted constant scalar curvature {K}\"ahler metrics with a large twist},
      JOURNAL = {},
       VOLUME = {},
     PAGES= {ArXiv:1508.00513},
       YEAR={2015},

}



\bib{Hong2}{article}{
    AUTHOR = {Hong, Ying-Ji  },
     TITLE = {Constant {H}ermitian scalar curvature equations on ruled
              manifolds},
   JOURNAL = {J. Differential Geom.},
    VOLUME = {53},
      YEAR = {1999},
    NUMBER = {3},
     PAGES = {465--516},

}

\bib{Hong3}{article}{
    AUTHOR = {Hong, Ying-Ji       },
     TITLE = {Gauge-fixing constant scalar curvature equations on ruled
              manifolds and the {F}utaki invariants},
   JOURNAL = {J. Differential Geom.},
    VOLUME = {60},
      YEAR = {2002},
    NUMBER = {3},
     PAGES = {389--453},
}

\bib{Ke1}{article}{
 AUTHOR={Keller, Julien}, 
  TITLE= {About canonical {K}\"ahler metrics on {M}umford semistable projective bundles over a curve},
  JOURNAL= {J. London Math. Soc.},
    VOLUME = {93},
    YEAR = {2016},
    NUMBER = {1},
      PAGES = {159-174}
      }
      
\bib{Kob}{book}{
AUTHOR={S. Kobayashi},
TITLE={Differential geometry of complex vector bundles},
VOLUME={Princeton University Press},
YEAR={1987}
}

\bib{MR2109686}{book}{
	author={Kodaira, Kunihiko},
	title={Complex manifolds and deformation of complex structures},
	series={Classics in Mathematics},
	edition={Reprint of the 1986 English edition},
	note={Translated from the 1981 Japanese original by Kazuo Akao},
	publisher={Springer-Verlag, Berlin},
	date={2005},
	pages={x+465},
	isbn={3-540-22614-1},
}

\bib{MR1274118}{article}{
   author={LeBrun, C.},
   author={Simanca, S. R.},
   title={Extremal K\"ahler metrics and complex deformation theory},
   journal={Geom. Funct. Anal.},
   volume={4},
   date={1994},
   number={3},
   pages={298--336},
   issn={1016-443X},
}

\bib{MR3313212}{article}{
    AUTHOR = {Li, Chi},
     TITLE = {Remarks on logarithmic {K}-stability},
   JOURNAL = {Commun. Contemp. Math.},
    VOLUME = {17},
      YEAR = {2015},
    NUMBER = {2},
     PAGES = {1450020, 17},
      ISSN = {0219-1997},
}
\bib{MR3248054}{article}{
    AUTHOR = {Li, Chi},
    Author = {Sun, Song},
     TITLE = {Conical {K}\"ahler-{E}instein metrics revisited},
   JOURNAL = {Comm. Math. Phys.},
    VOLUME = {331},
      YEAR = {2014},
    NUMBER = {3},
     PAGES = {927--973},
      ISSN = {0010-3616},
}



\bib{LiH}{article}{
    AUTHOR = {Li, Haozhao},
     TITLE = {Extremal {K}\"ahler metrics and energy functionals on projective
              bundles},
   JOURNAL = {Ann. Global Anal. Geom.},
    VOLUME = {41},
      YEAR = {2012},
    NUMBER = {4},
     PAGES = {423--445},
      ISSN = {0232-704X},
}


\bib{MR1701135}{article}{
   author={Li, Jiayu},
   author={Narasimhan, M. S.},
   title={Hermitian-Einstein metrics on parabolic stable bundles},
   journal={Acta Math. Sin. (Engl. Ser.)},
   volume={15},
   date={1999},
   number={1},
   pages={93--114},
}
\bib{LiJ}{article}{
    AUTHOR = {Li, Jiayu},
     TITLE = {Hermitian-{E}instein metrics and {C}hern number inequalities
              on parabolic stable bundles over {K}\"ahler manifolds},
   JOURNAL = {Comm. Anal. Geom.},
    VOLUME = {8},
      YEAR = {2000},
    NUMBER = {3},
     PAGES = {445--475},
}

	
\bib{LiL}{article}{
    AUTHOR = {Li, Long},
    TITLE = {Subharmonicity of conic Mabuchi's functional, I},
      JOURNAL = {},
       VOLUME = {},
     PAGES= {ArXiv:1511.00178.},
       YEAR={},

}



\bib{LZ}{article}{
    AUTHOR = {Li, Long},
    AUTHOR = {Zheng, Kai},
    TITLE = {A continuity method approach to the uniqueness of K\"ahler-Einstein cone metrics},
      JOURNAL = {},
       VOLUME = {},
     PAGES= {ArXiv:1511.02410.},
       YEAR={},

}

\bib{LZ2}{article}{
    AUTHOR = {Li, Long},
    AUTHOR = {Zheng, Kai},
    TITLE = {Uniqueness of constant scalar curvature K\"ahler metrics with cone singularities, I: Reductivity},
      JOURNAL = {},
       VOLUME = {},
     PAGES= {ArXiv:1603.01743},
       YEAR={},

}

\bib{LZ3}{article}{
    AUTHOR = {Li, Long},
    AUTHOR = {Zheng, Kai},
    TITLE = {Uniqueness of constant scalar curvature K\"ahler metrics with cone singularities, II: Bifurcation},
      JOURNAL = {},
       VOLUME = {},
     PAGES= {Preprint},
       YEAR={},

}




\bib{LS}{article}{
    AUTHOR = {Lu, Zhiqin},
    AUTHOR = {Seyyedali, Reza},
    TITLE = {Extremal metrics on ruled manifolds},
      JOURNAL = {Adv. Math.},
       VOLUME = {258},
     PAGES= {127-153},
       YEAR={2014},
}


\bib{MR3403730}{article}{
    AUTHOR = {Odaka, Yuji},
    AUTHOR = {Sun, Song},
     TITLE = {Testing log {K}-stability by blowing up formalism},
   JOURNAL = {Ann. Fac. Sci. Toulouse Math. (6)},
    VOLUME = {24},
      YEAR = {2015},
    NUMBER = {3},
     PAGES = {505--522},
   
}

\bib{MR3107685}{article}{
    AUTHOR = {Seyyedali, Reza},
     TITLE = {Balanced metrics and {C}how stability of projective bundles
              over {K}\"ahler manifolds {II}},
   JOURNAL = {J. Geom. Anal.},
    VOLUME = {23},
      YEAR = {2013},
    NUMBER = {4},
     PAGES = {1944--1975},
      ISSN = {1050-6926},
}

\bib{Sim}{article}{
    AUTHOR = {Simpson, Carlos T.},
     TITLE = {Constructing variations of {H}odge structure using
              {Y}ang-{M}ills theory and applications to uniformization},
   JOURNAL = {J. Amer. Math. Soc.},
    VOLUME = {1},
      YEAR = {1988},
    NUMBER = {4},
     PAGES = {867--918},
}

\bib{MR1040197}{article}{
   author={Simpson, Carlos T.},
   title={Harmonic bundles on noncompact curves},
   journal={J. Amer. Math. Soc.},
   volume={3},
   date={1990},
   number={3},
   pages={713--770},
}

\bib{MR3470713}{article}{
    AUTHOR = {Song, J.},
    Author = {Wang, X.},
     TITLE = {The greatest {R}icci lower bound, conical {E}instein metrics
              and {C}hern number inequality},
   JOURNAL = {Geom. Topol.},
    VOLUME = {20},
      YEAR = {2016},
    NUMBER = {1},
     PAGES = {49--102},
      ISSN = {1465-3060},
}
%
 \bib{MR1005085}{article}{
    author={Troyanov, Marc}, 
    title={Prescribing curvature on compact surfaces with conical
    singularities},
    journal={Trans. Amer. Math. Soc.},
    volume={324},
    date={1991},
    number={2},
    pages={793--821},
 }

\bib{MR861491}{article}{
   author={Uhlenbeck, K.},
   author={Yau, S.-T.},
   title={On the existence of Hermitian-Yang-Mills connections in stable
   vector bundles},
   note={Frontiers of the mathematical sciences: 1985 (New York, 1985)},
   journal={Comm. Pure Appl. Math.},
   volume={39},
   date={1986},
   number={S, suppl.},
   pages={S257--S293},
}

	
	
\bib{YZ}{article}{
    AUTHOR = {Yin, Hao},
    AUTHOR = {Zheng, Kai},
    TITLE = {Expansion formula for complex Monge-Amp\`ere equation along cone
singularities},
      JOURNAL = {},
       VOLUME = {},
     PAGES= {ArXiv:1609.03111},
       YEAR={},

}	

\bib{zhengcscKcone}{article}{
   author={Zheng, Kai},
   title={K\"ahler metrics with cone singularities and uniqueness problem},
   conference={
      title={Proceedings of the 9th ISAAC Congress, Krak\'ow 2013},
   },
   book={
   title={Current Trends in Analysis and its Applications},
      series={Trends in Mathematics},
      publisher={Springer International Publishing},
   },
   date={2015},
   pages={395-408},
}




\end{biblist}
\end{bibdiv}
\end{document}